\documentclass[a4paper, 12pt, oneside]{amsart}
\usepackage[utf8]{inputenc}
\usepackage{amsmath,amsthm,amsfonts,amssymb}
\usepackage{url}

\usepackage{graphicx} 

\usepackage{mathtools} 

\usepackage[shortlabels]{enumitem}

\usepackage{upgreek}

\usepackage{tabularx}
\usepackage{tabularray}
\usepackage{geometry}
\usepackage[dvipsnames]{xcolor}
\usepackage{graphicx}
\usepackage[shortlabels]{enumitem}
\usepackage{float}
\usepackage{tikz-cd}
\usetikzlibrary{arrows}

\definecolor{cvut}{RGB}{0, 101, 189}
\definecolor{uab}{RGB}{0, 133, 57}

\makeatletter
\newcommand{\leqnomode}{\tagsleft@true}
\newcommand{\reqnomode}{\tagsleft@false}
\makeatother

\DeclareMathOperator{\ske}{skew}

\newcommand\numberthis{\addtocounter{equation}{1}\tag{\theequation}}
\newcommand{\rest}[2]{\left.#1\right|_{#2}}

\newcommand{\dif}{\mathrm{d}}

\DeclareMathOperator{\en}{\mathrm{End}}

\newcommand{\e}{\mathrm{e}}

\newcommand{\pr}{\mathrm{pr}}

\newcommand{\gder}{\mathrm{gDer}}

\newcommand{\jac}{\mathrm{Jac}}

\newcommand{\cg}{\textsl{g}}

\newcommand\varlist {,\makebox[1em][c]{.\hfil.\hfil.},}

\newcommand{\nb}{_{\scalebox{0.45}{$\nabla$}}}

\newcommand{\s}[1]{\scalebox{0.5}{#1}}

\newcommand{\R}{{\mathbb{R}}}

\newcommand{\cC}{\mathcal{C}}

\newcommand{\cCi}{\cC^\infty}

\newcommand{\la}{\langle}
\newcommand{\ra}{\rangle}

\DeclareMathOperator{\rk}{rk}

\DeclareMathOperator{\im}{im}
\DeclareMathOperator{\id}{id}

\DeclareMathOperator{\Sym}{Sym}

\newtheorem{theorem}{Theorem}[section]
\newtheorem{corollary}[theorem]{Corollary}
\newtheorem{lemma}[theorem]{Lemma}
\newtheorem{proposition}[theorem]{Proposition}
\newtheorem{example}[theorem]{Example}

\newtheorem{innercustomthm}{Theorem}
\newenvironment{customthm}[1]
{\renewcommand\theinnercustomthm{#1}\innercustomthm}
{\endinnercustomthm}

\newtheorem{innercustomcor}{Corollary}
\newenvironment{customcor}[1]
{\renewcommand\theinnercustomcor{#1}\innercustomcor}
{\endinnercustomcor}

\theoremstyle{definition}
\newtheorem{definition}[theorem]{Definition}
\newtheorem*{definition*}{Definition}

\newtheorem{innercustomdef}{Definition}
\newenvironment{customdef}[1]
{\renewcommand\theinnercustomdef{#1}\innercustomdef}
{\endinnercustomdef}

\theoremstyle{remark}

\newtheorem{remark}[theorem]{Remark}

\title[]{Courant algebroid lifts\\
and curved Courant algebroids}

\author[F. Moučka]{Filip Moučka}
\author[R. Rubio]{Roberto Rubio}

\address{Universitat Aut\`onoma de Barcelona, 08193 Barcelona, Spain;\newline \indent Faculty of Nuclear Sciences and Physical Engineering, Czech Technical\newline\indent University in Prague, 115 19 Prague 1, Czech Republic}
\email{filip.moucka@autonoma.cat; mouckfil@cvut.cz}

\address{Universitat Aut\`onoma de Barcelona, 08193 Barcelona, Spain}
\email{roberto.rubio@uab.es}

\thanks{This project has been supported by MICIU/AEI/10.13039/501100011033 and the EU FEDER under the grants PID2022-137667NA-I00 (GENTLE) and CNS2024-154695 (DÉCOLLAGE). The first author has been partially supported by the Grant Agency of the Czech Technical University in Prague, grant No. SGS25/163/OHK4/3T/14. The second author has also received support from the MICIU/AEI and the EU FSE under the Ramón y Cajal fellowship RYC2020-030114-I and from the AGAUR under the grant 2021-SGR-01015.
}

\begin{document}

\begin{abstract}
We introduce the Courant algebroid lift, a new construction that takes a Courant algebroid together with a vector bundle connection and produces, when the connection is flat in the image of the anchor, a Courant algebroid. In general, this lift produces a Courant-like structure that we call a curved Courant algebroid. We start by establishing a hierarchy of Courant algebroid properties and their associated structures. In this setting, we introduce curved Courant algebroids, which we show to be related to connections with torsion and curved differential graded Lie algebras. We use this to provide a classification of exact curved Courant algebroids. We show that the Courant algebroid lift of an exact Courant algebroid yields a natural link between the Patterson-Walker metric and generalized geometry. By lifting non-exact Courant algebroids, we establish a relation of these lifts to Lie algebras, Poisson and special complex geometry. Finally, we show that Courant algebroid lifts provide a large class of examples of Courant algebroid actions.
\end{abstract}

\maketitle

\section{Introduction}
Courant algebroids \cite{CouDM, LiuMTLB} offer a very helpful framework for the study of geometric structures and equations in mathematical physics. For instance, they have been recently used to find an obstruction to the existence of solutions of the Hull–Strominger system \cite{futaki-inv}. Its definition makes use of the reduction of an exact Courant algebroid over the total space of a principle $G$-bundle, which does not produce a Courant algebroid unless the first Pontryagin class vanishes \cite[Let. 4]{SevL}. Thus, some constructions involving a Courant algebroid as the starting data naturally end up with a relaxed Courant-like structure. On the other hand, Courant algebroids have also been used to fully describe $\mathcal{N}=1$ supergravity in ten dimensions \cite{FricoDDS}. In this description, the Riemann curvature tensor of a Courant algebroid connection \cite{CouVys} plays a key role. Such tensor is well defined even in a weaker setting than that of a Courant algebroid, where some results can be naturally extended, as observed, for instance, for \cite[Thm. 3.1]{Palatini}. This illustrates another general phenomenon: theory originally developed in the setting of Courant algebroids remains meaningful in a broader context. Courant algebroid lifts and curved Courant algebroids, the main topics of this work, are natural instances of these two phenomena. 

Our first aim is to study and establish a precise hierarchy of the various properties satisfied by Courant algebroids (Figure \ref{CApropfig}), from which different notions of Courant algebroids (ante, pre, and the usual one) have arisen in the literature. This study motivates the introduction of a new class of structures.
\begin{customdef}{\ref{def: curved-ca}}
 We call a metric algebroid $(\mathbb{E},\la\,\,,\,\ra,\rho,[\,\,,\,])$ a \textbf{curved Courant algebroid} if $F(a,b)\coloneqq \rho([a,b])-[\rho(a),\rho(b)]_\text{Lie}$ satisfies:
\begin{itemize}
\item $F\in\Gamma(\wedge^2\mathbb{E}^*\otimes TM)$ and $F(a,b)=0$ for $a\in\ker\rho$.
\end{itemize}
We call $F$ the \textbf{curvature} of the curved Courant algebroid.
\end{customdef}
\noindent A first remarkable feature of curved Courant algebroids is that, when the algebroid is regular, $F$ is written as $F(a,b)=T(\rho(a),\rho(b))$ for some $T\in\Omega^2(M,TM)$, which is a torsion-like tensor field on the base manifold. In the transtive case, $T$ is unique.

Our second aim is to lay the foundations for the study of curved Courant algebroids. We begin by introducing \textit{curved Cartan calculus}, an extension of Cartan calculus depending on the choice of $T\in\Omega^2(M,TM)$. The \textit{curved exterior derivative} $\dif^{\s{$T$}}=\dif-\iota_T$ plays a generating role, thus giving the curved Lie derivative $L^{\s{$T$}}$ and the curved bracket $[\,\,,\,]^{\s{$T$}}$ on vector fields. The curvature $\mathcal{R}\coloneqq \dif^{\s{$T$}}\circ\dif^{\s{$T$}}$ is non-vanishing for $T\neq 0$ and the following graded-commutation relations hold:
 \begin{align*}
[\iota_X,\iota_Y]_\text{g}&=0, &[\iota_X,\dif^{\s{$T$}}]_\text{g}&=L^{\s{$T$}}_X, & [\dif^{\s{$T$}}, L^{\s{$T$}}_X]_\text{g}&=[\mathcal{R},\iota_X]_\text{g},\\
 [\dif^{\s{$T$}},\dif^{\s{$T$}}]_\text{g}&=2\mathcal{R}, & [L^{\s{$T$}}_X,\iota_Y]_\text{g}&=\iota_{[X,Y]^{\s{$T$}}}, & [L^{\s{$T$}}_X,L^{\s{$T$}}_Y]_\text{g}&=L_{[X,Y]^{\s{$T$}}}+[[\iota_X,\mathcal{R}]_\text{g},\iota_Y]_\text{g}.
 \end{align*}
Curved Cartan calculus allows us to show a close relation of curved Courant algebroids to connections with torsion and curved differential graded Lie algebras. Moreover, it provides a natural framework for describing exact curved Courant algebroids, for which we give a complete classification: those with different curvatures are not isomorphic, whereas for those with the same curvature, we obtain:
\begin{customthm}{\ref{thm: curved-class}}
We have the one-to-one correspondence
\begin{equation*}
\frac{\Omega^3(M)}{\im\rest{\dif^{\s{$T$}}}{\Omega^2(M)}} \cong\left\{\begin{array}{c}
 \text{exact curved Courant algebroids over }M\\
 \text{with curvature $T\in\Omega^2(M,TM)$}
 \end{array}\right\}\Big/\sim.
\end{equation*}
 \end{customthm}
\noindent Namely, every exact curved Courant algebroid is isomorphic to $TM\oplus T^*M$ with the canonical symmetric pairing $\la\,\,,\,\ra_+$, the anchor being the projection on $TM$, and the bracket, for $X+\alpha, Y+\beta\in\Gamma(TM\oplus T^*M)$, given by
\begin{equation*}
 [X+\alpha,Y+\beta]^{\s{$T$}}_{\s{$H$}}\coloneqq [X,Y]^{\s{$T$}}+L^{\s{$T$}}_X\beta-\iota_Y\dif^{\s{$T$}}\alpha+\iota_Y\iota_XH
\end{equation*}
for some $T\in\Omega^2(M,TM)$ and $H\in\Omega^3(M)$. This bracket also motivates an extension of the theory of derived brackets \cite{KosDB} 
(Section \ref{sec: derived-brackets}). 

Finally, we use this classification to show that ante, curved, pre and usual Courant algebroids are strictly different classes (Corollary \ref{cor: strict}). For better readability, most of the proofs of Sections \ref{sec: hierarchy} and \ref{sec: exact-curved} are in Appendix \ref{app:proofs}.

Our third and main aim is the definition of the \textit{Courant algebroid lift}. This is a construction that takes as an input a Courant-like structure $(\mathbb{E},\la\,\,,\,\ra,\rho,[\,\,,\,])$ over $M$, where $\mathbb{E}$ is a vector bundle of the form $\mathbb{E}=TM\oplus E$ for some vector bundle $E$, together with a vector bundle connection $\nabla$ on $E$ (note that many Courant algebroids are of this form, including, but not limited to, transitive Courant algebroids). The output of this construction is a Courant-like structure $(TE,g\nb,\rho\nb,[\,\,,\,]\nb)$ over $E$ given by the unique extensions of
\begin{align*}
 g\nb(\phi a,\phi b)&\coloneqq \pr^*\la a, b\ra, & \rho\nb(\phi a)&\coloneqq\phi\rho(a), & [\phi a,\phi b]\nb&\coloneqq \phi [a,b],
\end{align*}
where $a,b\in\Gamma(\mathbb{E})$ and the map $\phi\colon\Gamma(\mathbb{E})\rightarrow\mathfrak{X}(E)$ is induced by the connection $\nabla$ via the horizontal lift on $\mathfrak{X}(M)$ and the vertical lift on $\Gamma(E)$.

We first show that the lift of a curved Courant algebroid with curvature $F$ is a curved Courant algebroid, whose curvature depends on both $F$ and the curvature $R\nb$ of the connection $\nabla$ (Proposition \ref{prop: ca-lift-curved-ca}). We then prove one of our main results.

\begin{customthm}{\ref{thm: curved-ca}}
 The lift of a curved Courant algebroid is a Courant algebroid if and only if $R\nb(\rho(a),\rho(b))=0$ and the original tuple is Courant.
\end{customthm}
\noindent The resulting (curved) Courant algebroid is particularly interesting because it lives on the tangent bundle of a manifold, which allows to study the interaction between ordinary and generalized geometry. Moreover, it is never transitive if $E\neq 0$.

Theorems \ref{thm: curved-class} and \ref{thm: curved-ca} justify the terminology and relevance of curved Courant algebroids. First, exact structures are closely related to curved differential graded Lie algebras, where the vanishing of the curvature corresponds to recovering an ordinary Courant algebroid. Second, curved Courant algebroids arise naturally as lifts of Courant algebroids, with their curvature being determined by the curvature of the underlying vector bundle connection.

The most important example of a Courant algebroid lift is the case of the fundamental Courant algebroid, in which $\mathbb{E}=TM\oplus T^*M$ and $E=T^*M$ (Section \ref{sec: PW-Ca}). The vector bundle connection is the dual of an affine connection on $M$ and the \mbox{(pseudo-)Riem}annian metric we recover on $T^*M$ is precisely the \textit{\mbox{Patterson-Walk}er metric} \cite{PatRE}. The concept of Courant algebroid lifts thus establishes a natural, though previously hidden, bridge between the Patterson-Walker metric and generalized geometry \cite{HitGCYM, GuaGCG}. This is the example that originally led us to consider the Courant algebroid lift and then curved Courant algebroids.

More generally, for lifts of transitive Courant algebroids, the connection must be flat in order to obtain a Courant algebroid. 
Particularly interesting is the so-called \textit{$B_n$-Courant algebroid} living on $TM\oplus T^*M\oplus (M\times\R)$, whose Courant algebroid lift motivates the introduction of the \textit{odd Patterson-Walker metric}. Moreover, we show that every \textit{special complex structure} \cite{AleSCM} induces a Courant algebroid structure on $T(T^*M\oplus (M\times \R))$ by the Courant algebroid lift (Corollary \ref{cor: special-complex}).

Another relevant example comes from Poisson geometry. A Poisson bivector determines, via the double of a Lie bialgebroid construction \cite{LiuMTLB}, a Courant algebroid. By Theorem \ref{thm: curved-ca}, its lift by the dual of an affine connection on $M$ becomes a Courant algebroid if and only if the connection is flat on Hamiltonian vector fields. We use this to show that even the lift by a non-flat connection can be a Courant algebroid (Example \ref{ex: Poisson-Heisenberg}).


Finally, we show that there is a very close relation between Courant algebroid lifts and \textit{Courant algebroid actions} in the sense of \cite{LiBlCAPG}.
\begin{customthm}{\ref{thm: nabla-ca-action}}
Let $(\mathbb{E},\la\,\,,\,\ra,\rho,[\,\,,\,])$ such that $\mathbb{E}=TM\oplus E$ for some vector bundle $\pr\colon E\rightarrow M$ be a Courant algebroid, and let $\nabla$ be a vector bundle connection on $E$. The pair $(\pr,\rho\nb\circ\phi )$ is a Courant algebroid action of $(\mathbb{E},\la\,\,,\,\ra,\rho,[\,\,,\,])$ on $E$ if and only if the corresponding Courant algebroid lift is a Courant algebroid. If this is the case, the stabilizer of the action at an arbitrary point of $E$ is coisotropic. 
\end{customthm}

\noindent Courant algebroid lifts thus offer a wide range of examples of Courant algebroid actions with coisotropic stabilizers. For instance, we have the following:

\begin{customcor}{\ref{cor: exact-action}}
 An exact Courant algebroid (regarded as $\mathbb{T}M$) acts on $T^*M$ by flat affine connections on $M$.
\end{customcor}

\begin{customcor}{\ref{cor: special-complex-action}}
 \sloppy The $B_n$-Courant algebroid (Section \ref{sec: Bn}) acts on \mbox{$T^*M\oplus (M\times\R)$} by special complex structures on $M$.
\end{customcor}

\begin{customcor}{\ref{cor: Poisson-action}}
 The Courant algebroid induced by a Poisson structure on $M$ (Section \ref{sec: Poisson}) acts on $T^*M$ by affine connections on $M$ that are flat when restricted to Hamiltonian vector fields.
\end{customcor}

In future work, we would like to focus on the compatibility of the Courant algebroid lift construction with the additional structures living on a Courant algebroid like \textit{Dirac structures}, \textit{generalized complex structures}, \textit{generalized metrics}, \textit{Courant algebroid connections}... Another interesting question is the relevance of this construction in physics. Namely, its interaction with \textit{Poisson-Lie T-duality}, whose relation to Courant algebroid actions was noticed in \cite{SevCou}.

\bigskip\bigskip\bigskip

\textbf{Notation and conventions:} We work in the smooth category on a manifold $M$. By $X,Y,Z\in\mathfrak{X}(M)$ we denote vector fields. We use $\alpha,\beta,\eta\in\Omega^1(M)$ for $1$-forms, and $f,\cg\in\cCi(M)$ for functions.

By $\en(\Omega^\bullet(M))$ we denote the $\mathbb{Z}$-graded vector space of endomorphisms of the vector space of differential forms $\Omega^\bullet(M)$. We will make use of the fact that the graded commutator $[A,B]_\text{g}\coloneqq A\circ B-(-1)^{|A||B|}B\circ A$, for $A,B\in\en(\Omega^\bullet(M))$ of degrees $|A|, |B|$, gives $\en(\Omega^\bullet(M))$ the structure of a graded Lie algebra. 

\clearpage

\section{Hierarchy of structures generalizing Courant algebroids}\label{sec: hierarchy}
All the structures we will study share the following shape.
\begin{definition}
	By a \textbf{tuple} $(\mathbb{E},\la\,\,,\,\ra,\rho,[\,\,,\,])$ we mean a $4$-tuple consisting of:
	\begin{itemize}
		\item a vector bundle $\mathbb{E}\rightarrow M$,
		\item a non-degenerate symmetric $2$-form $\langle\,\,,\,\rangle\in\Gamma(\Sym^2\mathbb{E}^*)$,\hfill(the \textbf{pairing})
		\item a vector bundle morphism $\rho\colon \mathbb{E}\rightarrow TM$ over $\id_M$,\hfill(the \textbf{anchor})
		\item an $\mathbb{R}$-bilinear map $[\,\,,\,]\colon\Gamma(\mathbb{E})\times\Gamma(\mathbb{E})\rightarrow\Gamma(\mathbb{E})$.\hfill(the \textbf{bracket})
	\end{itemize}
\end{definition}
We will consistently denote the sections of $\mathbb{E}$ in a tuple by $a,b,c\in\Gamma(\mathbb{E})$.

\subsection{Courant and metric algebroids}

Our definition of tuple is motivated by the theory of Courant algebroids.
\begin{definition}\label{def: ca}
	A \textbf{Courant algebroid} is a tuple $(\mathbb{E},\la\,\,,\,\ra,\rho,[\,\,,\,])$ satisfying
	\begin{enumerate}[(\text{Ca}1),leftmargin=2cm]
		\item $[a,b]+[b,a]=\rho^*\dif \langle a,b\rangle$,\label{ca1}
		\item $\rho(a)\langle b,c\rangle=\langle [a,b],c\rangle+\langle b,[a,c]\rangle$,\label{ca2}
		\item $[a,[b,c]]=[[a,b],c]+[b,[a,c]]$,\label{ca3}
	\end{enumerate}
	where $\rho^*\colon T^*M\rightarrow \mathbb{E}$ is the vector bundle morphism given by $\rho^*\coloneqq \langle\,\,,\,\rangle^{-1}\circ\rho^t$.
\end{definition}
The fundamental example of a Courant algebroid is the generalized tangent bundle $\mathbb{T}M\coloneqq TM\oplus T^*M$, whose sections we denote by $X+\alpha, Y+\beta$.
\begin{example}\label{ex: fund-Ca}
The \textbf{fundamental Courant algebroid} is the tuple consisting of
\begin{itemize}
	\item the bundle $\mathbb{T}M$,
	\item the canonical symmetric pairing $\langle X+\alpha,Y+\beta\rangle_+\coloneqq \alpha(Y)+\beta(X)$,
	\item the projection on the tangent bundle $\rho\coloneqq \mathrm{pr}_{TM}$,	
 	\item the Dorfman bracket $[X+\alpha, Y+\beta]\coloneqq [X,Y]_\emph{Lie}+L_X\beta-\iota_Y\dif \alpha$.
\end{itemize}
\end{example}

We consider the more general framework of \textit{metric algebroids} \cite[Def. 2.1]{VaiGDFT}.

\begin{definition}
A \textbf{metric algebroid} is a tuple $(\mathbb{E},\langle\,\,,\,\rangle,\rho,[\,\,,\,])$ satisfying \ref{ca1} and \ref{ca2}.
\end{definition}

This setup already determines how the bracket interacts with the \mbox{$\cCi(M)$-mod}ule structure of $\Gamma(\mathbb{E})$. 

\begin{lemma}\label{lem: ca-4-5}
For a metric algebroid $(\mathbb{E},\la\,\,,\,\ra,\rho,[\,\,,\,])$ we have that\emph{
\begin{enumerate}[(\text{Ca}1),leftmargin=2cm]
\setcounter{enumi}{3}
\item $[a,fb]=(\rho(a)f)b+f[a,b]$,\label{ca4}
\item $[fa,b]=-(\rho(b)f)a+f[a,b]+\langle a,b\rangle\rho^*\dif f$.\label{ca5}
\end{enumerate}}
\end{lemma}

\begin{proof} See Appendix \ref{app:proofs}.
\end{proof}

A triple $(\mathbb{E},\rho,[\,\,,\,])$ satisfying \ref{ca3} and \ref{ca4} is called a \textit{Leibniz algebroid} \cite{IbaLA}. By \cite{BarLeib}, its anchor is always an algebra morphism when seen~as
\begin{equation*}
 \rho\colon (\Gamma(\mathbb{E}),[\,\,,\,])\rightarrow(\mathfrak{X}(M),[\,\,,\,]_\text{Lie}).
\end{equation*}
By Lemma \ref{lem: ca-4-5}, any Courant algebroid is a Leibniz algebroid, so it satisfies
\begin{enumerate}[(\text{Ca}6),leftmargin=2cm]
\setcounter{enumi}{5}
\item $\rho([a,b])=[\rho(a),\rho(b)]_{\mathrm{Lie}}$.\label{ca6}
\end{enumerate}

\subsection{Pre and ante-Courant algebroids}\label{sec: prop-Ca}
The fact that \ref{ca6} follows from the definition of a Courant algebroid motivates the following concept \cite[Def. 1.3]{VaipreCA}.

\begin{definition}
A \textbf{\mbox{pre-Cou}rant algebroid} is a metric algebroid satisfying \ref{ca6}.
\end{definition}

For a metric algebroid, there are two equivalent formulations of \ref{ca6} that describe the bracket of an arbitrary element with $\im \rho^*$.

\begin{lemma}\label{lem: ca-7-8}
A metric algebroid $(\mathbb{E},\la\,\,,\,\ra,\rho,[\,\,,\,])$ is pre-Courant if and only if one of the following equivalent properties is true
\emph{\begin{enumerate}[(\text{Ca}1),leftmargin=2cm]
\setcounter{enumi}{6}
\item $[a,\rho^*\alpha]=\rho^*L_{\rho(a)}\alpha$,\label{ca7}
\item $[\rho^*\alpha,a]=-\rho^*\iota_{\rho(a)}\dif \alpha$.\label{ca8}
\end{enumerate}}
\end{lemma}

\begin{proof}See Appendix \ref{app:proofs}.
\end{proof}

We obtain other algebroid structures by replacing \ref{ca6} with a weaker version of it, which we state in terms of the map $F\colon \Gamma(\mathbb{E})\times\Gamma(\mathbb{E})\rightarrow\mathfrak{X}(M)$ given by
\begin{equation}\label{eq: ante-curvature}
F(a,b)\coloneqq \rho([a,b])-[\rho(a),\rho(b)]_\text{Lie}.
\end{equation}

\begin{definition}
\sloppy A metric algebroid $(\mathbb{E},\la\,\,,\,\ra,\rho,[\,\,,\,])$ is an \textbf{ante-Courant algebroid} if
\begin{enumerate}[(\text{Ca}1),leftmargin=2cm]
\setcounter{enumi}{11}
\item $F$ 
is $\cCi(M)$-bilinear and skew-symmetric, that is, $F\in\Gamma(\wedge^2\mathbb{E}^*\otimes TM)$\label{ca12}
\end{enumerate}
\noindent (we leave a gap in the numbering as we introduce an intermediate notion below).
\end{definition}

Ante-Courant algebroids appear in the literature in the context of \textit{double field theory} with a different definition \cite[Def. A.31]{ChatanteCA}. We show their equivalence.

\begin{lemma}\label{lem: ca-13}
 A metric algebroid $(\mathbb{E},\la\,\,,\,\ra,\rho,[\,\,,\,])$ is ante-Courant if and only~if
 \emph{\begin{enumerate}[(\text{Ca}1),leftmargin=2cm]
\setcounter{enumi}{12}
\item $\rho\circ\rho^*=0$.\label{ca13}
\end{enumerate}}
\end{lemma}

\begin{proof}
 See Appendix \ref{app:proofs}.
\end{proof}

Condition \ref{ca13} has, on any tuple, two equivalent formulations in terms of the linear subsets $\im\rho^*,\ker\rho\subseteq \mathbb{E}$, which are not necessarily subbundles.

\begin{proposition}\label{prop: coisotropic}
A tuple $(\mathbb{E},\la\,\,,\,\ra,\rho,[\,\,,\,])$ satisfies \emph{\ref{ca13}} if and only if one of the following equivalent properties holds
 \begin{itemize}
 \item $\im\rho^*\subseteq \mathbb{E}$ is isotropic with respect to the pairing.
 \item $\ker\rho\subseteq \mathbb{E}$ is coisotropic with respect to the pairing.
 \end{itemize}
\end{proposition}

\begin{proof}
See Appendix \ref{app:proofs}.
\end{proof}

We see that ante-Courant algebroid is a stricter notion than metric algebroid. 

\begin{example}
    For $\mathbb{E}=T\R$, any (positive-definite) pairing and anchor are determined by the choice of $\cg,f\in\cCi(\R)$ so that
    \begin{align*}
        \la\partial_x,\partial_x\ra&\coloneqq 2\e^\cg, & \rho(\partial_x)&\coloneq f\partial_x.
    \end{align*}
    By \emph{\ref{ca1}}, the bracket $[\,\,,\,]$ of a metric algebroid $(T\R,\la\,\,,\,\ra,\rho,[\,\,,\,])$ is determined by
    \begin{equation*}
        [\partial_x,\partial_x]=\frac{1}{2}\rho^*\dif\la \partial_x,\partial_x\ra=\rho^*(\cg'\e^\cg\dif x)=f\cg'\partial_x.
    \end{equation*}
    The second axiom \emph{\ref{ca2}} of a metric algebroid is always satisfied
    \begin{equation*}
        2\la[\partial_x,\partial_x],\partial_x\ra=\la\rho^*\dif\la\partial_x,\partial_x\ra,\partial_x\ra=\rho(\partial_x)\la\partial_x,\partial_x\ra.
    \end{equation*}
    On the other hand, we have $\rho(\rho^*\dif x)=f^2\e^{-\cg}\partial_x$, that is, by Lemma \ref{lem: ca-13}, a metric algebroid on $T\R$ is ante-Courant if and only if $\rho=0$. 
\end{example}

\subsection{Curved Courant algebroids}
We introduce an intermediate notion between pre and ante-Courant algebroids. This notion is more connected to the geometry of the underlying manifold and arises naturally in the context of Courant algebroid lifts, which we introduce in Section \ref{sec: Ca-lift}. 

\begin{definition}\label{def: curved-ca}
We call a metric algebroid $(\mathbb{E},\la\,\,,\,\ra,\rho,[\,\,,\,])$ a \textbf{curved Courant algebroid} if $F$ as in \eqref{eq: ante-curvature} satisfies:
\begin{enumerate}[(\text{Ca}1),leftmargin=2cm]
\setcounter{enumi}{8}
\item $F\in\Gamma(\wedge^2\mathbb{E}^*\otimes TM)$ and $F(a,b)=0$ for $a\in\ker\rho$.\label{ca9}
\end{enumerate}
We call $F$ the \textbf{curvature} of the curved Courant algebroid.
\end{definition}

\begin{remark}
    If the curvature of a curved Courant algebroid vanishes ($F=0$), the structure does not become a Courant algebroid in general, but it is pre-Courant.
\end{remark}

The relevance of curved Courant algebroids is made apparent when we characterize them in the \textit{regular} and \textit{transitive} cases, as the curvature is given by a torsion-like tensor field on the base manifold.

\begin{definition}
A tuple is called \textbf{regular} if the rank of $\rho$ is constant, and \textbf{transitive} if $\rho$ is surjective.
\end{definition}

\begin{proposition}\label{prop: reg-trans-curved-ca}
 A regular metric algebroid $(\mathbb{E},\la\,\,,\,\ra,\rho,[\,\,,\,])$ is a curved Courant algebroid if and only if there is $T\in\Omega^2(M,TM)$ such that $F(a,b)=T(\rho(a),\rho(b))$. If the curved Courant algebroid is, in addition, transitive, such $T$ is given uniquely.
\end{proposition}

\begin{proof}
 If $F(a,b)=T(\rho(a),\rho(b))$ for some $T\in\Omega^2(M,TM)$, then \ref{ca9} follows immediately. Conversely, for a regular curved Courant algebroid $(\mathbb{E},\la\,\,,\,\ra,\rho,[\,\,,\,])$, we have that $\ker\rho\leq \mathbb{E}$ is a subbundle, so we can find a complement $L\leq \mathbb{E}$ such that $\mathbb{E}=\ker\rho\oplus L$. The anchor $\rho$ can be thus seen as a vector bundle isomorphism $L\rightarrow\im\rho$, hence we can identify $\rest{F}{L}$ with a section of $\wedge^2(\im\rho)^*\otimes TM$. The regularity of $\rho$ gives that $\im\rho\leq TM$ is a subbundle, so we can again find a complement $C\leq TM$ such that $TM=\im\rho\oplus C$. We then extend $F\in\Gamma(\wedge^2(\im\rho)^*\otimes TM$) to $T\in\Omega^2(M,TM)$ simply by asking $T(u,v)=0$ if $u\in C$. The result follows from the fact that $F(a,b)=0$ if $a\in\ker\rho$. The claim about the transitive case is straightforward.
\end{proof}

\begin{remark}\label{rk: transitive-curvature}
 In light of Proposition \ref{prop: reg-trans-curved-ca}, if a curved Courant algebroid is transitive, we refer to both $F\in\Gamma(\wedge^2\mathbb{E}^*\otimes TM)$ and $T\in\Omega^2(M,TM)$ as its curvature.
\end{remark}
 
We have an analogue of Lemma \ref{lem: ca-7-8} for curved Courant algebroids. 

\begin{lemma}\label{lem: ca-10-11}
 A metric algebroid $(\mathbb{E},\la\,\,,\,\ra,\rho,[\,\,,\,])$ is a curved Courant algebroid if and only if one of the following equivalent properties is satisfied
\emph{\begin{enumerate}[(\text{Ca}1),leftmargin=2cm]
\setcounter{enumi}{9}
\item $F\in\Gamma(\wedge^2\mathbb{E}^*\otimes TM)$ and $[\ker \rho,\im\rho^*]=0$.\label{ca10}
\item $F\in\Gamma(\wedge^2\mathbb{E}^*\otimes TM)$ and $[\im\rho^*,\ker \rho]=0$.\label{ca11}
\end{enumerate}}
\end{lemma}

\begin{proof}
See Appendix \ref{app:proofs}.
\end{proof}

In the regular case we can phrase Lemma \ref{lem: ca-10-11} differently.

\begin{lemma}
\label{lem: reg-ca-10-11}
 A regular metric algebroid $(\mathbb{E},\la\,\,,\,\ra,\rho,[\,\,,\,])$ is a curved Courant algebroid if and only if one of the following equivalent statements is true
\begin{itemize}
\item there is $T\in\Omega^2(M,TM)$ such that $[a,\rho^*\alpha]=\rho^*(L_{\rho(a)}-(\iota_{\rho(a)}T)^t)\alpha$.
\item there is $T\in\Omega^2(M,TM)$ such that $[\rho^*\alpha,a]=\rho^*(-\iota_{\rho(a)}\circ\dif+(\iota_{\rho(a)}T)^t)\alpha$.
\end{itemize}
\end{lemma}

\begin{proof}
 See Appendix \ref{app:proofs}.
\end{proof}

In summary, we have the hierarchy of structures weakening the notion of Courant algebroid, but still staying metric algebroids.

\begin{corollary}\label{cor: inclusions}
There is the sequence of inclusions:
\small{\begin{equation*}
\left\{\begin{array}{c}
\text{Courant}\\
 \text{algebroids}\\
 \text{\footnotesize{\emph{\ref{ca3}}}}\end{array}\right\}\subseteq\left\{\begin{array}{c}
\text{pre-Courant}\\
 \text{algebroids}\\
 \text{\footnotesize{\emph{\ref{ca6}}}}\end{array}\right\}\subseteq\left\{\begin{array}{c}
\text{curved Courant}\\
 \text{algebroids}\\
 \text{\footnotesize{\emph{\ref{ca9}}}}\end{array}\right\}\subseteq\left\{\begin{array}{c}
\text{ante-Courant}\\
 \text{algebroids}\\
 \text{ \footnotesize{\emph{\ref{ca12}}}}\end{array}\right\}.
\end{equation*}}
Namely, every Courant algebroid possesses all properties \emph{\ref{ca1}} – \emph{\ref{ca13}}.
\end{corollary}

Figure \ref{CApropfig} depicts the hierarchy of Courant algebroid properties. In particular, it shows schematically how to derive the dependent properties from the three original Courant algebroid axioms, which are shown vertically.

\begin{remark}
\sloppy Other notions can be understood within this framework, like \textit{\mbox{pre-DFT} algebroids} \cite{ChatanteCA}, which are tuples satisfying only \ref{ca2}.
\end{remark}

\begin{figure}[ht]
\begin{center}
\resizebox{0.9\textwidth}{!}{\begin{tikzpicture}
 \node (1) at (8.5, 0) {\textcolor{uab}{\ref{ca1}}};
 \node (2) at (8.5, -3.5) {\textcolor{uab}{\ref{ca2}}};
 \node (3) at (8.5, -7) {\ref{ca3}};
 \node (4) at (11, -5.25) {\ref{ca4}};
 \node (5) at (11, -1.75) {\ref{ca5}};
 \node (6) at (6.5, -5.25){\ref{ca6}};
 \node (7) at (6.5, -3.5) {\ref{ca7}};
 \node (8) at (6.5, -1.75) {\ref{ca8}};
 \node (9) at (3.5, -5.25) {\ref{ca9}};
 \node (10) at (3.5, -3.5) {\ref{ca10}};
 \node (11) at (3.5, -1.75) {\ref{ca11}};
 \node (12) at (0.5, -5.25) {\ref{ca12}};
 \node (13) at (0.5, 0){\ref{ca13}};
 \node(14) at (8.5, -7.45) {\textcolor{cvut}{C}ourant \textcolor{cvut}{a}lgebroid};
 \node(15) at (6.5, -5.7) {\textcolor{cvut}{pre-Ca}};
 \node(16) at (3.5, -5.7) {\textcolor{cvut}{curved Ca}};
 \node(17) at (0.5, -5.7) {\textcolor{cvut}{ante-Ca}};

 \draw [->] (2) -- (4);

 \draw [->] (1) -- (5);
 \draw (4) to [out=90,in=153.43] (9.75, -0.875);

 \draw[->] (4) -- (6);
 \draw (3) to [out=90,in=0] (7.5, -5.25);

 \draw[->] (6) -- (9);

 \draw[implies-implies,double equal sign distance] (6) -- (7);
 \draw[->] (2) to [out=200, in=0] (6.6,-4.375);
 
 \draw[implies-implies,double equal sign distance] (7) -- (8);
 \draw[->] (1) to [out=225, in=0] (6.6,-2.625);

 \draw[implies-implies,double equal sign distance] (9) -- (10);
 \draw (2) to (7.45,-3.5);
 \draw[->] (5.55,-3.5) to [out=180,in=0] (3.6,-4.375);
 \draw[dotted, thick] (7.4,-3.5) to (7);
 \draw[dotted, thick] (5.6,-3.5) to (7);

 \draw[implies-implies,double equal sign distance] (10) -- (11);
 \draw[->] (1) to [out=190,in=0] (3.6,-2.625);
 
 \draw[->] (9) -- (12);

 \draw[implies-implies,double equal sign distance] (12) -- (13);
 \draw[->] (1) to [out=180, in=0] (0.6,-1.25);
\draw (4) to [out=60,in=270] (11.5, 0);
\draw (11.5, 0) to [out=90,in=0] (8.5, 2);
\draw (8.5, 2) to [out=180,in=90] (5.5, 0);
\draw (5.5, 0) to [out=270,in=12.5] (5.25, -0.3425);

 \draw[Plum, dashed, thick] (-4,-4.5) -- (12,-4.55);
 \draw[Plum, dashed, thick] (-4,-6.2) -- (12,-6.2);
 \draw[Orange, dashed, thick] (-4,-6.3) -- (12,-6.3);
 \draw[OrangeRed, dashed, thick] (0.5,0) circle (0.8);
 
\node[text width=3.3cm, text centered, uab] at (8.5,1) {Metric algebroid\\ axioms};
 \node[text width=4.5cm, Orange] at (-1.75,-7) {involving \\only $[\,\,,\,]$};
 \node[text width=4.5cm, Plum] at (-1.75,-5.25) {involving \\only $\rho$ and $[\,\,,\,]$};
 \node[text width=4.5cm, OrangeRed] at (-1.75,0) {involving \\only $\la\,\,,\,\ra$ and $\rho$};
 \node[text width=4.5cm] at (-1.75,-2.625) {involving\\ $\la\,\,,\,\ra$, $\rho$, and $[\,\,,\,]$};
 
\end{tikzpicture}}
\end{center}
\caption{Hierarchy of Courant algebroid properties.}\label{CApropfig}
\end{figure}
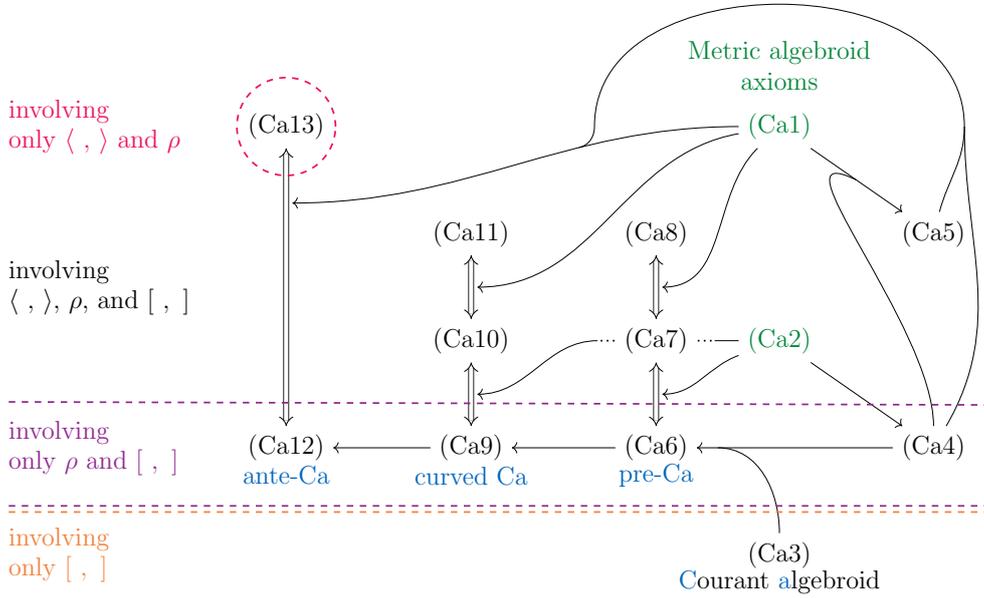


\section{Exact curved Courant algebroids}\label{sec: exact-curved}
In this section, we introduce \textit{curved Cartan calculus} on $\Omega^\bullet(M)$ and use it to classify exact curved Courant algebroids and reveal their relation to \textit{curved differential graded Lie algebras}.

\subsection{Curved Cartan calculus} \sloppy Graded derivations of the exterior algebra $(\Omega^\bullet(M),\wedge)$, that is, $D\in\en(\Omega^\bullet(M))$ such that, for $\psi,\varphi\in\Omega^\bullet(M)$,
\begin{equation*}
 D(\psi\wedge\varphi)=D\psi\wedge\varphi+(-1)^{|\psi||D|}\psi\wedge D\varphi
\end{equation*}
form a graded Lie subalgebra of $\en(\Omega^\bullet(M)$ with respect to the graded commutator $[\,\,,\,]_\text{g}$, which we will denote by $\gder(M)$. The classical Cartan calculus on $\Omega^\bullet(M)$ is described by the relations between distinguished elements of $\gder(M)$:
\begin{equation}\label{eq: Cartan-calculus}
 \begin{aligned}
 [\iota_X,\iota_Y]_\text{g}&=0, &&&&& [\iota_X,\dif]_\text{g}&=L_X, &&&&& [\dif, L_X]_\text{g}&=0,\\
 [\dif,\dif]_\text{g}&=0, &&&&& [L_X,\iota_Y]_\text{g}&=\iota_{[X,Y]_\text{Lie}}, & &&&&[L_X,L_Y]_\text{g}&=L_{[X,Y]_\text{Lie}}.
 \end{aligned}
\end{equation}
The exterior derivative $\dif$ is the unique element of $\gder_1(M)$ such that $(\dif f)(X)=Xf$ and $\dif\circ \dif=\frac{1}{2}[\dif,\dif]_\text{g}=0$. Alternatively, $\dif$ is given by the explicit formula using an auxiliary torsion-free connection $\nabla$ on $M$: for $\psi\in\Omega^r(M)$,
\begin{equation}\label{eq: dif-con}
 \dif\psi=(r+1)\ske (\nabla\psi).
\end{equation} 

If we omit the torsion-freeness of $\nabla$ in \eqref{eq: dif-con}, we still get a well-defined operator in $\gder_1(M)$ coinciding with $\dif$ on functions. Given two connections $\nabla$ and $\nabla'$, the corresponding operators are the same if and only if they coincide on $1$-forms, that~is,
\begin{align*}
 0&=2\ske(\nabla\alpha-\nabla'\alpha)(X,Y)=\alpha(-\nabla_XY+\nabla_YX+\nabla'_XY-\nabla'_YX)\\
 &=-\alpha(T\nb(X,Y)-T_{\scalebox{0.45}{$\nabla'$}}(X,Y)),
\end{align*}
where $T\nb,T_{\scalebox{0.45}{$\nabla'$}}\in\Omega^2(M,TM)$ stand for the torsions of $\nabla$ and $\nabla'$. Therefore, the operators coincide if and only if their underlying connections have the same torsion. As for a given $T\in\Omega^2(M,TM)$ we can always find a connection with torsion $T$, the following definition is justified.

\begin{definition}\label{def: curved-d}
Let $T\in\Omega^2(M, TM)$ and $\nabla$ be any connection on $M$ with torsion $T$. We introduce the \textbf{curved exterior derivative} $\dif^{\s{$T$}}$ as the unique element of $\gder_1(M)$ such that $\dif^{\s{$T$}}f=\dif f$ and $\dif^{\s{$T$}}\alpha=2\ske(\nabla\alpha)$ for $f\in\cCi(M)$ and $\alpha\in\Omega^1(M)$.
\end{definition}

Explicitly, we can express the curved exterior derivative $\dif^{\s{$T$}}$ without referring to an auxiliary connection as
\begin{equation}\label{eq: curved-d}
 \dif^{\s{$T$}}=\dif-\iota_T,
\end{equation}
where $\iota_\varphi$, for $\varphi\in\Omega^k(M,TM)$, is a generalization of the contraction operator (see e.g. \cite{MichRFN}) given as the unique element of $\gder_{k-1}(M)$ such that
\begin{align*}
    \iota_\varphi f&=0, & (\iota_\varphi\alpha)(X_1\varlist X_k)=\alpha(\varphi(X_1\varlist X_k)).
\end{align*}

Indeed, we have that $\dif^{\s{$T$}}$ and $\dif-\iota_T$ coincide on the functions and, moreover, for an auxiliary connection $\nabla$ with torsion $T$, we get
\begin{align*}
 (\dif^{\s{$T$}}\alpha)(X,Y)&=(\nabla_X\alpha)(Y)-(\nabla_Y\alpha)(X)=X\alpha(Y)-Y\alpha(X)-\alpha(\nabla_XY-\nabla_YX)\\
 &=X\alpha(Y)-Y\alpha(X)-\alpha([X,Y]_\text{Lie})-\alpha(T(X,Y))\\
 &=((\dif-\iota_T)\alpha)(X,Y).
\end{align*}

There is an elegant characterization of curved exterior derivatives. Consider an arbitrary $D\in\gder_1(M)$ such that $Df=\dif f$, hence
\begin{equation*}
 (\dif-D)(f\alpha)=f(\dif-D)(\alpha),
\end{equation*}
that is, $\dif-D\in\gder_1(M)$ is uniquely determined by a $\cCi(M)$-linear map $\Omega^1(M)\rightarrow\Omega^2(M)$, which is, by the contraction, canonically identified with an element in $\Omega^2(M,TM)$. Therefore, $D=\dif^{\s{$T$}}$ for some $T\in\Omega^2(M,TM)$ and we have the one-to-one correspondence:
\begin{equation*}
\Omega^2(M,TM) \cong\{\text{$D\in\gder_1(M)$ such that $Df=\dif f$}
 \},
\end{equation*}
where $0\in \Omega^2(M,TM)$ is sent to $\dif$. In particular, $\dif^{\s{$T$}}\circ \dif^{\s{$T$}}= 0$ if and only if $T= 0$.

Given an operator $\dif^{\s{$T$}}$, we introduce the \textbf{curved Lie derivative} by
\begin{equation}\label{eq: Lie-der}
L^{\s{$T$}}_X\coloneqq[\iota_X,\dif^{\s{$T$}}]_\text{g}=\iota_X\circ\dif^{\s{$T$}}+\dif^{\s{$T$}}\circ\iota_X
\end{equation}
and, for any connection $\nabla$ with torsion $T$, the \textbf{curved bracket}
\begin{equation*}
[X,Y]^{\s{$T$}}\coloneqq\nabla_XY-\nabla_YX=T(X,Y)+[X,Y]_\text{Lie}. 
\end{equation*}
It is easy to check that it can be derived from $\dif^{\s{$T$}}$ by
\begin{equation}\label{eq: bracket}
 \iota_{[X,Y]^{\s{$T$}}}=[[\iota_X,\dif^{\s{$T$}}]_\text{g},\iota_Y]_\text{g}=[L_X^{\s{$T$}},\iota_Y]_\text{g}.
\end{equation}
We shall see that every $T\in\Omega^2(M,TM)$ endows $(\en(\Omega^\bullet(M)),[\,\,,\,]_\text{g})$ with the structure of a \textit{curved differential graded Lie algebra}, for which we extend the definition in \cite{MauCDGLA}.

\begin{definition}
 A \textbf{curved differential graded Lie algebra} is a tuple $(\mathfrak{g},[\,\,,\,],D,\mathcal{R})$ consisting of a $\mathbb{Z}$-graded Lie algebra $(\mathfrak{g},[\,\,,\,])$, a \textbf{differential}: \mbox{$D\in\gder_\text{odd}(\mathfrak{g},[\,\,,\,])$}, and a \textbf{curvature}: $\mathcal{R}\in\mathfrak{g}_\text{ev}$, such that, for $u\in\mathfrak{g}$,
 \begin{align}\label{eq: cdgLa}
 D\mathcal{R}&=0, & &\text{and} & &D(D u)=[\mathcal{R},u].
 \end{align}
\end{definition}

\begin{remark}
 By the graded Jacobi identity, for every $v\in \mathfrak{g}_\text{odd}$, we have that $D\coloneqq[v,\,]\in\gder_\text{odd}(\mathfrak{g},[\,\,,\,])$. In this case, \eqref{eq: cdgLa} takes the following form:
 \begin{align*}
 [v,\mathcal{R}]&=0, & &\text{and} & &[v,[v,u]]=\frac{1}{2}[[v,v],u]=[\mathcal{R},u].
 \end{align*}
 Therefore, the curvature can always be chosen as $\mathcal{R}=\frac{1}{2}[v,v]$.
\end{remark}

Thus, given $T\in\Omega^2(M,TM)$, the structure of a curved differential graded Lie algebra on $(\en(\Omega^\bullet (M)),[\,\,,\,]_\text{g})$ is given by the differential $[\dif^{\s{$T$}},\,]_\text{g}$ and the curvature $\mathcal{R}\coloneqq\frac{1}{2}[\dif^{\s{$T$}},\dif^{\s{$T$}}]_\text{g}=\dif^{\s{$T$}}\circ\dif^{\s{$T$}}\in\gder_2(M)$. Explicitly, using \eqref{eq: curved-d}, we find that
 \begin{equation}\label{eq: curvature}
 \mathcal{R}=-\iota_{\jac^{\s{$T$}}}-[\dif,\iota_{T}]_\text{g},
 \end{equation}
 where $\jac^{\s{$T$}}(X,Y,Z)\coloneqq T(X,T(Y,Z))-T(T(X,Y),Z)-T(Y,T(X,Z))$.

By \eqref{eq: Lie-der} and \eqref{eq: bracket} and the graded Jacobi identity for $[\,\,,\,]_\text{g}$, we arrive to a generalization of the classical Cartan calculus \eqref{eq: Cartan-calculus}:
\begin{equation}\label{eq: curved-Cartan-calculus}
 \begin{aligned}
 [\iota_X,\iota_Y]_\text{g}&=0, &[\iota_X,\dif^{\s{$T$}}]_\text{g}&=L^{\s{$T$}}_X, & [\dif^{\s{$T$}}, L^{\s{$T$}}_X]_\text{g}&=[\mathcal{R},\iota_X]_\text{g},\\
 [\dif^{\s{$T$}},\dif^{\s{$T$}}]_\text{g}&=2\mathcal{R}, & [L^{\s{$T$}}_X,\iota_Y]_\text{g}&=\iota_{[X,Y]^{\s{$T$}}}, & [L^{\s{$T$}}_X,L^{\s{$T$}}_Y]_\text{g}&=L_{[X,Y]^{\s{$T$}}}+[[\iota_X,\mathcal{R}]_\text{g},\iota_Y]_\text{g}.
 \end{aligned}
\end{equation}

\subsection{Classification of exact curved Courant algebroids} We call a metric algebroid \textbf{exact} if the following sequence is exact:
\begin{equation}\label{eq: ex-seq}
\begin{tikzcd}[column sep=2em]
0\arrow[r]&{T^*M}\arrow{r}{\rho^*}&\mathbb{E}\arrow{r}{\rho}& TM\arrow[r]&0.
\end{tikzcd}
\end{equation}
The Courant algebroid from Example \ref{ex: fund-Ca} is exact, as $\rho^*\colon T^*M\rightarrow \mathbb{T}M$ is the inclusion and $\rho\colon \mathbb{T}M\rightarrow TM$ is the projection. This is still the case if we replace the Dorfman bracket with the $H$-\textbf{twisted Dorfman bracket}
\begin{equation*}
[X+\alpha,Y+\beta]_{\s{$H$}}\coloneqq [X+\alpha,Y+\beta]+\iota_Y\iota_XH
\end{equation*}
for some closed $3$-form $H\in\Omega^3(M)$. In fact, $H^3_\text{dR}(M)$ classifies exact Courant algebroids up to isomorphism \cite{SevL}. We define this isomorphism notion in our generality, so that it applies to all Courant-like structures in Section \ref{sec: prop-Ca}.
\begin{definition}\label{def: iso}
 Given two metric algebroids $(\mathbb{E},\la\,\,,\,\ra,\rho,[\,\,,\,])$ and $(\mathbb{E}',\la\,\,,\,\ra',\rho',[\,\,,\,]')$ over $M$, a vector bundle isomorphism $\phi\colon \mathbb{E}\rightarrow \mathbb{E}'$ over the identity is called a (metric algebroid) \textbf{isomorphism} if
\begin{align*}
&\langle \phi a,\phi b\rangle'=\langle a,b\rangle, & &\rho'\circ\phi=\rho, & &[\phi a,\phi b]'=\phi[a,b].
\end{align*}
\end{definition}

We see how curved Cartan calculus provides examples of curved Courant algebroids. For $T\in\Omega^2(M,TM)$ and $H\in\Omega^3(M)$, we define the bracket $[\,\,,\,]^{\s{$T$}}_{\s{$H$}}$ by
 \begin{equation*}
[X+\alpha,Y+\beta]^{\s{$T$}}_{\s{$H$}}\coloneqq [X,Y]^{\s{$T$}}+L^{\s{$T$}}_X\beta-\iota_Y\dif^{\s{$T$}}\alpha+\iota_Y\iota_XH.
 \end{equation*}
Note that if $T=0$, the bracket $[\,\,,\,]^{\s{$T$}}_{\s{$H$}}$ becomes the $H$-twisted Dorfman bracket. On the other hand, for arbitrary $T$, $H$, and $\rho\coloneqq\pr_{TM}$, we still have 
 \begin{equation*}
 [X+\alpha,X+\alpha]^{\s{$T$}}_{\s{$H$}}=L^{\s{$T$}}_X\alpha-\iota_X\dif^{\s{$T$}}\alpha=\dif^{\s{$T$}}\iota_X\alpha=\dif\alpha(X)=\frac{1}{2}\rho^*\dif\la X+\alpha,X+\alpha\ra_+,
 \end{equation*}
 and polarization yields that $(\mathbb{T}M,\la\,\,,\,\ra_+,\pr_{TM},[\,\,,\,]^{\s{$T$}}_{\s{$H$}})$ satisfies \ref{ca1}. Moreover,
 \begin{align*}
 2\la[X+\alpha,Y+\beta]_{\s{$H$}}^{\s{$T$}},Y+\beta\ra_+&=2((L^{\s{$T$}}_X\beta)(Y)-(\dif^{\s{$T$}}\alpha)(Y,Y)+\beta([X,Y]^{\s{$T$}}))=2X\beta(Y)\\
 &=\rho(X+\alpha)\la Y+\beta,Y+\beta\ra_+,
 \end{align*}
 hence, by polarization, \ref{ca2} is satisfied. So, $(\mathbb{T}M,\la\,\,,\,\ra_+,\pr_{TM},[\,\,,\,]^{\s{$T$}}_{\s{$H$}})$ is a metric algebroid. Using the notation of \eqref{eq: ante-curvature} and since $\ker\rho=T^*M$, we see that
 \begin{equation*}
 F(X+\alpha,Y+\beta)=[X,Y]^{\s{$T$}}-[X,Y]_\text{Lie}=T(X,Y).
 \end{equation*}
 Therefore, the tuple $(\mathbb{T}M,\la\,\,,\,\ra_+,\pr_{TM},[\,\,,\,]^{\s{$T$}}_{\s{$H$}})$ is an exact curved Courant algebroid.

 \begin{remark}
 Note that an exact curved Courant algebroid is, in particular, transitive. Following Remark \ref{rk: transitive-curvature}, the curvature of $(\mathbb{T}M,\la\,\,,\,\ra_+,\pr_{TM},[\,\,,\,]^{\s{$T$}}_{\s{$H$}})$ is $T$, that is, the torsion of an underlying connection.
 \end{remark}

 We proceed with the classification of exact curved Courant algebroids.

\begin{proposition}\label{prop: curved-class}
Every exact curved Courant algebroid over $M$ is isomorphic to $(\mathbb{T}M,\la\,\,,\,\ra_+,\pr_{TM},[\,\,,\,]_{\s{$H$}}^{\s{$T$}})$ for some $T\in\Omega^2(M,TM)$ and $H\in\Omega^3(M)$.
\end{proposition}

\begin{proof} See Appendix \ref{app:proofs}.
\end{proof}

\sloppy Proposition \ref{prop: curved-class} gives that each equivalence class on the set of exact curved Courant algebroids on $M$, where the equivalence is given by the natural notion of isomorphism (Definition \ref{def: iso}), contains at least one curved Courant algebroid of the form $(\mathbb{T}M,\la\,\,,\,\ra_+,\pr_{TM},[\,\,,\,]_{\s{$H$}}^{\s{$T$}})$. In the next proposition, we show when two such algebroids are contained in the same class.

\begin{lemma}\label{lem: curved-class}
Two curved Courant algebroids $(\mathbb{T}M,\la\,\,,\,\ra_+,\pr_{TM},[\,\,,\,]_{\s{$H$}}^{\s{$T$}})$ and $(\mathbb{T}M,\la\,\,,\,\ra_+,\pr_{TM},[\,\,,\,]_{\s{$H'$}}^{\s{$T'$}})$ are isomorphic if and only if
\begin{align*}
T=T' & &\text{and} & &H'-H\in\im\rest{\dif^{\s{$T$}}}{\Omega^2(M)}.
\end{align*}
\end{lemma}

\begin{proof} See Appendix \ref{app:proofs}.
\end{proof}

From Proposition \ref{prop: curved-class} and Lemma \ref{lem: curved-class}, we deduce the classification of exact curved Courant algebroids: those with different curvatures are not isomorphic, and for those with the same curvature we have the following.

\begin{theorem}\label{thm: curved-class}
We have the one-to-one correspondence
\begin{equation*}
\frac{\Omega^3(M)}{\im\rest{\dif^{\s{$T$}}}{\Omega^2(M)}} \cong\left\{\begin{array}{c}
 \text{exact curved Courant algebroids over }M\\
 \text{with curvature $T\in\Omega^2(M,TM)$}
 \end{array}\right\}\Big/\sim.
\end{equation*}
\end{theorem}

Clearly, an exact curved Courant algebroid is pre-Courant if and only if its curvature vanishes, and we thus obtain the following for exact pre-Courant algebroids:
\begin{equation}\label{eq: class-pre}
\frac{\Omega^3(M)}{\Omega^3_\text{exact}(M)} \cong \{
 \text{exact pre-Courant algebroids over }M\}\Big/\sim.
\end{equation}

\begin{corollary}\label{cor: strict}
 The inclusions in Corollary \ref{cor: inclusions} are strict.
\end{corollary}

\begin{proof}
 The strictness of the first inclusion follows from the classification of exact (pre-)Courant algebroids (\cite{SevL} and \eqref{eq: class-pre}) and the fact that there are non-closed $3$-forms in dimension greater than three. Similarly, the strictness of the second one follows from \eqref{eq: class-pre}, Theorem \ref{thm: curved-class}, and the fact that $\Omega^2(M,TM)\neq 0$ in dimension greater than one. Finally, for a trivector field $\mathcal{X}\in\mathfrak{X}^3(M)$, consider the bracket 
 \begin{equation*}
 [X+\alpha,Y+\beta]_{\s{$\mathcal{X}$}}\coloneqq [X+\alpha,Y+\beta]+\iota_\beta\iota_\alpha\mathcal{X}.
 \end{equation*}
 It follows easily from the skew-symmetry of $\mathcal{X}$ and the fact that the fundamental Courant algebroid satisfies \ref{ca1} and \ref{ca2} that $(\mathbb{T}M,\la\,\,,\,\ra_+,\pr_{TM},[\,\,,\,]_{\s{$\mathcal{X}$}})$ is a metric algebroid. As $F(X+\alpha,Y+\beta)=\iota_\beta\iota_\alpha\mathcal{X}$, it is also an ante-Courant algebroid. However, if $\mathcal{X}\neq0$, which is always possible in dimension greater than two, it is not a curved Courant algebroid because $\ker\rho=T^*M$, which proves that the third inclusion is also strict.
\end{proof}

\subsection{Curved differential graded Lie algebras and derived brackets}\label{sec: derived-brackets} 
There is a natural way to see the vector space $\Gamma(\mathbb{T}M)$ inside $\en(\Omega^\bullet(M))$, by identifying $X+\alpha\in\Gamma(\mathbb{T}M)$ with $\iota_X+\alpha\wedge\in\en(\Omega^\bullet(M))$. This identification allows us to derive the bracket $[\,\,,\,]^{\s{$T$}}_{\s{$H$}}$ for arbitrary $T\in\Omega^2(M,TM)$ and $H\in\Omega^3(M)$. Namely, for $\dif_{\s{$H$}}^{\s{$T$}}\coloneqq\dif^{\s{$T$}}+H\wedge\in\en(\Omega^\bullet(M))$, we have
\begin{equation}\label{eq:derived-bracket-d-T-H}
 [X+\alpha,Y+\beta]^{\s{$T$}}_{\s{$H$}}=[[X+\alpha,\dif^{\s{$T$}}_{\s{$H$}}]_\text{g},Y+\beta]_\text{g}.
\end{equation}

This example shows a generalization of the concept of \textit{derived bracket} \cite{KosDB}: when the differential (in this case, $\dif^{\s{$T$}}_{\s{$H$}}$) does not square to zero, we get a bracket that does not necessarily satisfy the Jacobi identity (as the ones we find in pre, curved and ante-Courant algebroids). Moreover, we obtain a curved diffential graded Lie algebra. In the case of \eqref{eq:derived-bracket-d-T-H}, it is $(\en(\Omega^\bullet(M)),[\,\,,\,]_\text{g},[\dif_{\s{$H$}}^{\s{$T$}},\,]_\text{g},\mathcal{R}_{\s{$H$}})$, where
\begin{equation*}
 \mathcal{R}_{\s{$H$}}\coloneqq\dif_{\s{$H$}}^{\s{$T$}}\circ \dif_{\s{$H$}}^{\s{$T$}}=\mathcal{R}+(\dif^{\s{$T$}}H)\wedge
\end{equation*}
and $\mathcal{R}\coloneqq\dif^{\s{$T$}}\circ \dif^{\s{$T$}}$, which can be also expressed by \eqref{eq: curvature}.

We can compute the Jacobiator of $[\,\,,\,]^{\s{$T$}}_{\s{$H$}}$ in terms of $\mathcal{R}$ and $\dif^{\s{$T$}}H$. 

\begin{proposition}
    For $T\in\Omega^2(M,TM)$, $H\in\Omega^3(M)$, the Jacobiator of $[\,\,,\,]^{\s{$T$}}_{\s{$H$}}$ reads as
    \begin{equation*}
J(X,Y,Z)+[[\iota_X,\mathcal{R}]_\emph{g},\iota_Y]_\emph{g}\eta+\iota_Z[\mathcal{R},\iota_X]_\emph{g}\beta-\iota_Z\iota_Y\mathcal{R}\alpha+\iota_Z\iota_Y\iota_X\dif^{\s{$T$}}H,
 \end{equation*}
where $J(X,Y,Z)\in\mathfrak{X}(M)$ is given by $\iota_{J(X,Y,Z)}=[[[\iota_X,\mathcal{R}]_\emph{g},\iota_Y]_\emph{g},\iota_Z]_\emph{g}$. In particular, the Jacobi identity \emph{\ref{ca3}} is satisfied if and only if $T=0$ and $\dif H=0$.
\end{proposition}

\begin{proof}
    By \eqref{eq: curved-Cartan-calculus} and the graded Jacobi identity for $[\,\,,\,]_\text{g}$, we have that the Jacobiator 
    \[[X+\alpha,[Y+\beta, Z+\eta]^{\s{$T$}}_{\s{$H$}}]^{\s{$T$}}_{\s{$H$}}-[[X+\alpha, Y+\beta]^{\s{$T$}}_{\s{$H$}}, Z+\eta]^{\s{$T$}}_{\s{$H$}}-[Y+\beta,[X+\alpha, Z+\eta]^{\s{$T$}}_{\s{$H$}}]^{\s{$T$}}_{\s{$H$}}\] is equal to
 \begin{equation*}
J(X,Y,Z)+[[\iota_X,\mathcal{R}]_\text{g},\iota_Y]_\text{g}\eta+\iota_Z[\mathcal{R},\iota_X]_\text{g}\beta-\iota_Z\iota_Y\mathcal{R}\alpha+\iota_Z\iota_Y\iota_X\dif^{\s{$T$}}H.
 \end{equation*}
 The Jacobi identity \ref{ca3} is thus satisfied if and only if $\mathcal{R}$ and $\dif^{\s{$T$}}H$ vanish, so $\mathcal{R}_{\s{$H$}}$ vanishes, that is, the curved differential graded Lie algebra is flat. This of course happens if and only if $T=0$ and $\dif H=0$.
\end{proof}

Another instance of this generalization of the derived bracket formalism is the curved bracket \eqref{eq: bracket} on $\mathfrak{X}(M)$ with the underlying curved differential graded Lie algebra $(\en(\Omega^\bullet(M)),[\,\,,\,]_\text{g},[\dif^{\s{$T$}},\,]_\text{g},\mathcal{R})$.


\section{Courant algebroid lifts}\label{sec: Ca-lift}
 In this section, we present a natural construction that leads from a Courant algebroid to a curved Courant algebroid. The starting data is a tuple $(\mathbb{E},\la\,\,,\,\ra, \rho, [\,\,,\,])$ over $M$, where $\mathbb{E}= TM\oplus E$ for some vector bundle $\pr\colon E\rightarrow M$, together with a vector bundle connection $\nabla$ on $E$. Note that we do not assume yet any of the properties \ref{ca1} – \ref{ca13} to be satisfied by $(\mathbb{E},\la\,\,,\,\ra, \rho, [\,\,,\,])$.

The choice of a connection $\nabla$ corresponds to the decomposition of $TE$ into a horizontal subbundle $\mathcal{H}\nb$ and the vertical subbundle $\mathcal{V}$:
\begin{equation*}
 TE=\mathcal{H}\nb\oplus \mathcal{V}\cong \pr^* TM\oplus \pr^*E=\pr^*\mathbb{E}.
\end{equation*}
In addition, we have the injective $\cCi(M)$-module morphism
\begin{equation*}
 \phi\nb \colon \Gamma(\mathbb{E})\rightarrow\mathfrak{X}(E),
\end{equation*}
whose restriction to $\mathfrak{X}(M)$ and $\Gamma(E)$ is the horizontal lift $(\;)^\text{h}$ and the vertical lift $(\;)^\text{v}$ respectively (see Appendix \ref{app: lifts} for an overview of the various lifts that we use). The subset $\im\phi\nb $ locally generates the entire space $\mathfrak{X}(E)$ and, moreover, for every $q\in E$ we have the vector space isomorphism
\begin{equation}\label{eq: vs-iso}
\begin{split}
 \phi\nb^q\colon \mathbb{E}_{\pr(q)}&\rightarrow T_qE\\
\text{a}&\mapsto (\phi a)_q,
\end{split}
\end{equation}
where $a$ is an arbitrary section of $\mathbb{E}$ such that $a_{\pr(q)}=\text{a}$. To simplify the notation, the subscript $\nabla$ from $\mathcal{H}\nb$, $\phi\nb$, $\phi\nb^q$ will be omitted in the following sections.

\subsection{Definition of the Courant algebroid lift}\label{sec: construction}
Let $(\mathbb{E},\la\,\,,\,\ra, \rho, [\,\,,\,])$ and $\nabla$ be as above. We describe how to construct a tuple over $E$. 

First, we define the \mbox{(pseudo-)Riem}annian metric $g\nb$ on $E$ (the pairing on $TE$):
\begin{equation}\label{eq: metric}
 g\nb(\phi a,\phi b)\coloneqq\pr^*\la a,b\ra
\end{equation}
for $a,b\in\Gamma(\mathbb{E})$. Analogously, we lift the anchor $\rho$ and the bracket $[\,\,,\,]$ to the maps $\rho\nb\colon\im\phi \rightarrow \mathfrak{X}(E)$ and $[\,\,,\,]\nb\colon\im\phi \times\im\phi \rightarrow\im\phi $ by
\begin{align}\label{eq: rho-nb-hor}
 \rho\nb(\phi a)\coloneqq \phi \rho(a)=\rho(a)^\text{h} & &\text{and} & &[\phi a,\phi b]\nb\coloneqq \phi [a,b].
\end{align}
We show that $\rho\nb$ and $[\,\,,\,]\nb$ are naturally extended to an anchor and a bracket.

\begin{lemma}\label{lem: anchor}
 There is a unique extension of $\rho\nb$ to a $\cCi(E)$-linear map $\mathfrak{X}(E)\rightarrow\mathfrak{X}(E)$, so $\rho\nb$ is identified with a vector bundle endomorphism over the identity of $TE$. In addition, it satisfies
 \begin{equation}\label{eq: rho-phi}
 \rho^*\nb\circ\pr^*=\phi \circ\rho^*.
 \end{equation}
\end{lemma}

\begin{proof}
 For $f\in\cCi(M)$, we have, by Lemma \ref{lem: hor-ver} and \eqref{eq: rho-nb-hor}, that
\begin{equation*}
 \rho\nb ((\pr^*f)\phi a)=\rho\nb (\phi (fa))=\rho(fa)^\text{h}=(\pr^*f)\rho(a)^\text{h}=(\pr^*f)\rho\nb(\phi a),
\end{equation*}
hence asking the extension of $\rho\nb$ by $\cCi(E)$-linearity is meaningful. Since $\rho\nb$ is a pointwise operator and $\im\phi $ locally generates $\mathfrak{X}(E)$, there is indeed a unique extension of $\rho\nb$ to $\mathfrak{X}(E)$. By \eqref{eq: rho-nb-hor} and the fact that $\rho^*\nb=g\nb^{-1}\circ\rho\nb^t$, we find that
\begin{align*}
 g\nb(\rho^*\nb\pr^*\alpha,\phi a)&=(\pr^*\alpha)(\rho(a)^\text{h})=\pr^*(\alpha(\rho(a))=\pr^*\la\rho^*\alpha,a\ra=g\nb(\phi \rho^*\alpha,\phi a)
 \end{align*}
 and the result follows from the non-degeneracy of $g\nb$.
\end{proof}

\begin{lemma}\label{lem: brakcet}
 If $(\mathbb{E},\la\,\,,\,\ra, \rho, [\,\,,\,])$ satisfies \emph{\ref{ca4}} and \emph{\ref{ca5}}, there is a unique extension of $[\,\,,\,]\nb$ to a bracket on $TE$, which we denote by the same symbol $[\,\,,\,]\nb$, such that $(TE,g\nb,\rho\nb,[\,\,,\,]\nb)$ also satisfies \emph{\ref{ca4}} and \emph{\ref{ca5}}.
\end{lemma}

\begin{proof}

As $(\mathbb{E},\la\,\,,\,\ra, \rho, [\,\,,\,])$ satisfies \ref{ca4}, we find, by Lemma \ref{lem: hor-ver} and \eqref{eq: rho-nb-hor}, that
\begin{align*}
 [\phi a,(\pr^*f)\phi b]\nb&=[\phi a,\phi (fb)]\nb=\phi [a,fb]=\phi ((\rho(a)f)b+f[a,b])\\
 &=(\pr^*(\rho(a)f))\phi b+(\pr^*f)\phi [a,b]\\
 &=(\rho(a)^\text{h}\pr^*f)\phi b+(\pr^*f)[\phi a,\phi b]\nb\\
 &=(\rho\nb(\phi a)\pr^*f)\phi b+(\pr^*f)[\phi a,\phi b]\nb.
\end{align*}
Analogously, using that $(\mathbb{E},\la\,\,,\,\ra, \rho, [\,\,,\,])$ satisfies \ref{ca5} yields
\begin{equation*}
 [(\pr^*f)\phi a,\phi b]\nb=-(\rho\nb(\phi a)\pr^*f)\phi b+(\pr^*f)[\phi a,\phi b]\nb+(\pr^*\la a,b\ra)\phi (\rho^*\dif f ).
\end{equation*}
By \eqref{eq: rho-phi}, we finally arrive at
\begin{equation*}
 [(\pr^*f)\phi a,\phi b]\nb=-(\rho\nb(\phi a)\pr^*f)\phi b+(\pr^*f)[\phi a,\phi b]\nb+g\nb(\phi a,\phi b)\rho\nb^*\dif (\pr^*f ).
\end{equation*}
Therefore, asking \ref{ca4} and \ref{ca5} to be satisfied by the extension of $[\,\,,\,]\nb$ (together with $g\nb$ and $\rho\nb$) is meaningful and the result follows from the fact that $[\,\,,\,]\nb$ is a local operator in both inputs and that $\im\phi $ locally generates $\mathfrak{X}(E)$.
\end{proof}

We put all this together into a definition.

\begin{definition}\label{def: ca-lift}
 Let $(\mathbb{E},\la\,\,,\,\ra, \rho, [\,\,,\,])$ be a tuple over $M$ satisfying \ref{ca4} and \ref{ca5} such that $\mathbb{E}=TM\oplus E$ for some vector bundle $E$ over $M$, and let $\nabla$ be a vector bundle connection on $E$. We call the tuple $(TE,g\nb,\rho\nb,[\,\,,\,]\nb)$ constructed above the \textbf{Courant algebroid lift} of $(\mathbb{E},\la\,\,,\,\ra, \rho, [\,\,,\,])$ by $\nabla$.
\end{definition}

\begin{remark}\label{rk: ca-lift-transitive}
 Note that the Courant algebroid lift is always non-transitive if $E\neq 0$, as $\rk \rest{\rho\nb}{q}=\rk\rho_{\pr(q)}\leq\dim M<\dim E$ for all $q\in E$. On the other hand, it is regular if and only if the original tuple is regular. If $E=0$, the identification $E\cong M$ also identifies the lift with the original tuple. 
\end{remark}

We have seen that the Courant algebroid lift always satisfies \ref{ca4} and \ref{ca5}, just as the starting tuple does. In this section and the following we explore which properties from \ref{ca1} – \ref{ca13} are inherited by the Courant algebroid lift. We first need a technical lemma.

\begin{lemma}\label{lem: Ca45}
 Let $(\mathbb{E},\la\,\,,\,\ra,\rho,[\,\,,\,])$ satisfy only \emph{\ref{ca4}} and \emph{\ref{ca5}}. The maps $S\colon\Gamma(\mathbb{E})\times\Gamma(\mathbb{E})\rightarrow\Gamma(\mathbb{E})$ and $K\colon\Gamma(\mathbb{E})\times\Gamma(\mathbb{E})\times\Gamma(\mathbb{E})\rightarrow\cCi(M)$ given by
\begin{align*}
 S(a,b)&\coloneqq[a,b]+[b,a]-\rho^*\dif\la a,b\ra,\\
 K(a,b,c)&\coloneqq\rho(a)\la b,c\ra-\la[a,b],c\ra-\la b, [a,c]\ra,
\end{align*}
are $\cCi(M)$-linear in all their inputs. If the tuple is, in addition, a pre-Courant algebroid, the map $\jac_{[\,\,,\,]}\colon \Gamma(\mathbb{E})\times\Gamma(\mathbb{E})\times\Gamma(\mathbb{E})\rightarrow\Gamma(\mathbb{E})$,
\begin{equation*}
 \jac_{[\,\,,\,]}(a,b,c)\coloneqq[a,[b,c]]-[[a,b],c]-[b,[a,c]]
\end{equation*}
 is $\cCi(M)$-linear and skew-symmetric in all its inputs.
\end{lemma}

\begin{proof}
The claims about the maps $S$ and $K$ follow straightforwardly by using \ref{ca4} and \ref{ca5}. For the map $\jac_{[\,\,,\,]}$, we find similarly that
\begin{align*}
 \jac_{[\,\,,\,]}(a,b,fc)=\,&f\jac(a,b,c)-(F(a,b)f)c,\\
 \jac_{[\,\,,\,]}(a,a,b)=\,&-[S(a,a),b]-[\rho^*\dif\la a,a\ra,b],\\
 \jac_{[\,\,,\,]}(a, b,b)=\,&[a,S(b,b)]+[a,\rho^*\dif \la b,b\ra]-\rho^*L_{\rho(a)}\dif\la b,b\ra-S([a,b],b)+K(a,b,b).
\end{align*}
By definition, $(\mathbb{E},\la\,\,,\,\ra,\rho,[\,\,,\,])$ is pre-Courant if and only if $S=0$, $K=0$, and $F=0$. Moreover, by Lemma \ref{lem: ca-7-8}, we have that $[a,\rho^*\dif f]=\rho^*L_{\rho(a)}\dif f$ and $[\rho^*\dif f,a]=0$ for any pre-Courant algebroid, hence the result follows.
\end{proof}

\begin{lemma}\label{lem: ca1213}
The Courant algebroid lift satisfies the property \emph{\ref{ca1}}, \emph{\ref{ca2}}, or \emph{\ref{ca13}} if and only if the original tuple satisfies the same property.
\end{lemma}

\begin{proof}
Since, by Lemma \ref{lem: brakcet}, every Courant algebroid lift always satisfies \ref{ca4} and \ref{ca5}, it follows from the first part of Lemma \ref{lem: Ca45} that it is enough to check that the properties \ref{ca1} and \ref{ca2} are satisfied on local generators of $\mathfrak{X}(E)$.

As \eqref{eq: rho-phi} is true for $\rho\nb$, we find that
\begin{align*}
[\phi a,\phi b]+[\phi b,\phi a]-\rho\nb^*\dif g\nb(\phi a,\phi b)&=\phi ([a,b]+[b,a])-\rho\nb^*\pr^*\dif\la a,b\ra\\
&=\phi([a,b]+[b,a]-\rho^*\dif\la a,b\ra).
 \end{align*}
Analogously, \eqref{eq: rho-nb-hor} and Lemma \ref{lem: hor-ver} give that
 \begin{align*}
\rho\nb (\phi a)g\nb(\phi b,\phi c)-g\nb([\phi a,\phi b]\nb&,\phi c)-g\nb(\phi b,[\phi a, \phi c]\nb)\\
&=\pr^*(\rho(a)\la b,c\ra-\la[a,b],c\ra-\la b,[a,c]\ra).
 \end{align*}
The results for \ref{ca1} and \ref{ca2} follow from the injectivity of $\phi$ and $\pr^*$ respectively.

Finally, using the fact that we have the vector space isomorphism \eqref{eq: vs-iso} yields
 \begin{align*}
 0&=\rho\nb(\phi ^q\text{a})=\rest{\rho\nb(\phi a)}{q}=(\rest{\phi \rho(a))}{q}=\phi ^q\rho(\text{a}) & &\Leftrightarrow & 0&=\rho(\text{a})
 \end{align*}
for $q\in E$, $\text{a}\in\mathbb{E}_{\pr(q)}$ and $a\in\Gamma(\mathbb{E})$ such that $a_{\pr(a)}=\text{a}$, hence 
\begin{equation}\label{eq: kernels}
 \ker\rest{\rho\nb}{q}=\phi ^q(\ker\rho_{\pr(q)}).
\end{equation}
Therefore, $\phi ^q \text{a}\in(\ker\rest{\rho\nb}{q})^\perp$ if and only if, for $\text{b}\in\ker\rho_{\pr(q)}$, we have that
\begin{equation*}
 0= g\nb(\phi^q \text{a},\phi^q \text{b})=\la \text{a},\text{b}\ra,
\end{equation*}
that is, $(\ker\rest{\rho\nb}{q})^\perp=\phi^q((\ker\rho_{\pr(q)})^\perp)$. So $\ker\rest{\rho\nb}{q}$ is coisotropic if and only if $\ker\rho_{\pr(q)}$ is coisotropic. The result for \ref{ca13} follows from Proposition \ref{prop: coisotropic}.
\end{proof}

\subsection{Curved Courant algebroids as lifts of Courant algebroids}
Lemma \ref{lem: ca1213} gives that the lift of a metric (ante-Courant) algebroid is a metric (ante-Courant) algebroid. We will show that the same applies for curved Courant algebroids. In the following, we will denote the curvature of a vector bundle connection $\nabla$ by $R\nb\in\Omega^2(M,\en E)$, and we will also use the \textit{vertical lift of a field of endomorphisms $(\;)^\upsilon\colon \Gamma(\en E)\rightarrow\mathfrak{X}(E)$} (see Appendix \ref{app: lifts}).

\begin{proposition}\label{prop: ca-lift-curved-ca}
 The lift of a curved Courant algebroid with curvature $F$ is a curved Courant algebroid with curvature $F\nb\in\Omega^2(E,TE)$ given by
 \begin{equation*}
 F\nb (\phi a,\phi b)=F(a,b)^\emph{h}+R\nb(\rho(a),\rho(b))^\upsilon.
 \end{equation*}
\end{proposition}

\begin{proof}
 As every curved Courant algebroid is ante-Courant, its lift is, by Lemma \ref{lem: ca1213}, also an ante-Courant algebroid, that is, $F\nb\in\Omega^2(E,TE)$ and it is thus determined by its values on local generators. By \eqref{eq: rho-nb-hor}
 and the fact that $F$ is the curvature of the original curved Courant algebroid, we get
 \begin{align*}
 F\nb(\phi a,\phi b)&=\rho\nb([\phi a,\phi b]\nb)-[\rho\nb(\phi a),\rho\nb(\phi b)]_\text{Lie}=\rho([a,b])^\text{h}-[\rho(a)^\text{h},\rho(b)^\text{h}]_\text{Lie}\\
 &=[\rho(a),\rho(b)]^\text{h}_\text{Lie}+F(a,b)^\text{h}-[\rho(a)^\text{h},\rho(b)^\text{h}]_\text{Lie}.
 \end{align*}
Lemma \ref{lem: hor-ver-com} yields $F\nb(\phi a,\phi b)=F(a,b)^\text{h}+R\nb(\rho(a),\rho(b))^\upsilon$. The result follows by going pointwise and using \eqref{eq: kernels}.
\end{proof}

By Proposition \ref{prop: ca-lift-curved-ca}, even if $F$ vanishes, there is still a term involving $R\nb$ in the formula for $F\nb$. By the injectivity of $(\;)^\upsilon$, we deduce the following.

\begin{corollary}\label{cor: curved-pre-ca}
The lift of a pre-Courant algebroid is a curved Courant algebroid with curvature $F\nb\in\Omega^2(E,TE)$ given by
 \begin{equation*}
 F\nb (\phi a,\phi b)=R\nb(\rho(a),\rho(b))^\upsilon.
 \end{equation*}
In particular, the lift becomes pre-Courant if and only if $R\nb(\rho(a),\rho(b))=0$.
\end{corollary}

Finally, we characterize when the Courant algebroid lift is Courant.

\begin{theorem}\label{thm: curved-ca}
 The lift of a curved Courant algebroid is a Courant algebroid if and only if $R\nb(\rho(a),\rho(b))=0$ and the original tuple is Courant.
\end{theorem}

\begin{proof}
As the original tuple is a curved Courant algebroid, we get, by the second part of Lemma \ref{lem: Ca45}, that the Courant algebroid lift becomes a Courant algebroid if and only if $F\nb=0$ and the Jacobiator for $[\,\,,\,]\nb$ vanishes on local generators of $\mathfrak{X}(E)$. The result follows from Proposition \ref{prop: ca-lift-curved-ca}, the injectivity of $(\;)^\upsilon$ and $\phi $, and the fact that $\jac_{[\,\,,\,]\nb}(\phi a,\phi b,\phi c)=\phi \jac_{[\,\,,\,]}(a,b,c)=0$.
\end{proof}

\subsection{Lifts of exact tuples}\label{sec: PW-Ca}
The most prominent examples of Courant algebroid lifts are those, when $E=T^*M$ and $(\mathbb{E},\la\,\,,\,\ra,\rho,[\,\,,\,])$ is the fundamental Courant algebroid on $M$ (Example \ref{ex: fund-Ca}). A vector bundle connection on $E$ is then just the dual connection of an ordinary affine connection on $M$.

By Proposition \ref{prop: ca-lift-curved-ca}, the lift is a curved Courant algebroid. Since $\rho=\pr_{TM}$, the lifted anchor $\rho\nb$ is simply just the projection on the horizontal subbundle $\pr_{\mathcal{H}}$ and thus the curvature of $(T(T^*M),g\nb,\rho\nb,[\,\,,\,]\nb)$ coincides with the Riemann curvature $R\nb\in\Omega^2(M, \en TM)$. More precisely, since the curvature of the dual connection is given by $(X,Y)\mapsto -R\nb(X,Y)^t$, we have that
\begin{equation*} F\nb(X^\text{h}+\alpha^\text{v},Y^\text{h}+\beta^\text{v})=-(R\nb(X,Y)^t)^\upsilon.
\end{equation*}

We can also express the bracket $[\,\,,\,]\nb$ explicitly in this case. For $V,W\in\mathfrak{X}(T^*M)$
\begin{equation*}
 [V,W]\nb=\pr_{\mathcal{H}}[V_1,W_1]_\text{Lie}+g\nb^{-1}(L_{V_1}g\nb(W_2)-\iota_{W_1}\dif g\nb(V_2)),
\end{equation*}
where $V_1\coloneqq \pr_{\mathcal{H}}V$, $V_2\coloneqq \pr_{\mathcal{V}}V$ and $W_1,W_2$ are defined analogously. For local generators the formula can be simplified as follows:
\begin{equation*}
 [X^\text{h}+\alpha^\text{v},Y^\text{h}+\beta^\text{v}]\nb=[X,Y]^\text{h}_\text{Lie}+(L_{X}\beta-\iota_{Y}\dif \alpha)^\text{v}.
\end{equation*}

The (pseudo-)Riemannian metric $g\nb$ on $T^*M$ that we recovered by the Courant algebroid lift is the \textit{Patterson-Walker metric} \cite{PatRE}. This interesting object can be seen as a (pseudo-)Riemannian analogue of the canonical symplectic form on $T^*M$ \cite{SymCartan, SymPoisson}, and it finds applications e.g. in projective geometry \cite{dunajski-mettler}, conformal geometry \cite{conformal-PW} or para-Kähler geometry \cite{CapIPK}.

\begin{remark}
Two Patterson-Walker metrics coincide if and only if the underlying connections have the same associated torsion-free connection (see, e.g., \cite[Prop. 5.5]{SymCartan}). However, the horizontal subbundles, and hence the anchors $\rho\nb$, do not coincide in that case, so we do not restrict ourselves to torsion-free connections.
\end{remark}

 By Theorem \ref{thm: curved-ca}, the lift is a Courant algebroid if and only if $\nabla$ is flat. In this particular case, the resulting Courant algebroid on $T(T^*M)$ fits also into two different frameworks that have appeared in the literature.

 \subsubsection{Lie bialgebroid construction}
 A \textit{Lie algebroid} is a Leibniz algebroid whose bracket is anti-commutative \cite{PraFR}. Recall that Lie algebroids $(\mathbb{A},\rho_\mathbb{A}, [\,\,,\,]_\mathbb{A})$ and $(\mathbb{A}^*,\rho_{\mathbb{A}^*},[\,\,,\,]_{\mathbb{A}^*})$ are said to form a \textit{Lie bialgebroid} if, for $a,b\in\Gamma(\wedge^\bullet\mathbb{A})$, we have
 \begin{equation*}
     \dif_{\mathbb{A}^*}[a,b]=[ \dif_{\mathbb{A}^*}a,b]+[a, \dif_{\mathbb{A}^*}b],
 \end{equation*}
 where $\dif_{\mathbb{A}^*}$ is a Lie algebroid differential coming from $(\mathbb{A}^*,\rho_{\mathbb{A}^*},[\,\,,\,]_{\mathbb{A}^*})$ and $[\,\,,\,]$ is the natural extension of the Lie algebroid bracket on $\mathbb{A}$ to $\Gamma(\wedge^\bullet \mathbb{A})$.
 
 By \cite[Thm. 2.5]{LiuMTLB}, every Lie bialgebroid induces a Courant algebroid structure on $\mathbb{A}\oplus \mathbb{A}^*$, for $a+\chi, b+\tau\in\Gamma(\mathbb{A}\oplus \mathbb{A}^*)$, given by
 \begin{align*}
     \la a+\chi, b+\tau\ra&\coloneq \chi(a)+\tau(b),\\
     \rho&\coloneqq \rho_{\mathbb{A}}\circ \pr_{\mathbb{A}}+\rho_{\mathbb{A}^*}\circ \pr_{\mathbb{A}^*},\\
     [a+\chi,b+\tau]&\coloneqq[a,b]_\mathbb{A}+L^{\mathbb{A}^*}_\chi b-\iota_\tau\dif_{\mathbb{A}^*}a+[\chi,\tau]_{\mathbb{A}^*}+L^\mathbb{A}_a\tau-\iota_b\dif_\mathbb{A}\chi,
 \end{align*}
 where $L^\mathbb{A}$ and $\dif_\mathbb{A}$ denote the Lie derivative and differential on $\Gamma(\wedge^\bullet\mathbb{A}^*)$ associated with the Lie algebroid structure $(\mathbb{A},\rho_\mathbb{A}, [\,\,,\,]_\mathbb{A})$, and $L^{\mathbb{A}^*}$ and $\dif_{\mathbb{A}^*}$ are defined analogously for $(\mathbb{A}^*,\rho_{\mathbb{A}^*},[\,\,,\,]_{\mathbb{A}^*})$. In particular, for any Lie algebroid $(\mathbb{A},\rho_\mathbb{A}, [\,\,,\,]_\mathbb{A})$, we can construct a Lie bialgebroid by choosing the trivial Lie algebroid structure on $\mathbb{A}^*$, that is, $\rho_{\mathbb{A}^*}=0$ and $[\,\,,\,]_{\mathbb{A}^*}=0$. 
 
 Given a connection $\nabla$ on $M$, we can identify the vertical subbundle $\mathcal{V}\leq T(T^*M)$ with the dual of the horizontal subbundle $\mathcal{H}^*$ using the Patterson-Walker metric $g\nb$ on $T^*M$. It is easy to check that $\mathcal{H}$ together with the inclusion $\mathcal{H}\hookrightarrow T(T^*M)$ and the restriction of $[\,\,,\,]_\text{Lie}$ on $\mathfrak{X}(T^*M)$ to $\Gamma(\mathcal{H})$ forms a Lie algebroid if and only if $\nabla$ is flat. Choosing the trivial Lie algebroid structure on $\mathcal{H}^*$ makes $(\mathcal{H},\mathcal{H}^*)$ a Lie bialgebroid. The lift $(T(T^*M),g\nb,\rho\nb,[\,\,,\,]\nb)$ is clearly isomorphic to the Courant algebroid on $\mathcal{H}\oplus\mathcal{H}^*$ obtained from the Lie bialgebroid structure.
 
\subsubsection{The Courant algebroid of a para-Hermitian foliation}
A \textit{para-Hermitian structure} on a manifold $Q$ is a pair $(g,A)$ consisting of a \mbox{(pseudo-)Riem}annian metric $g$ on $Q$ and $A\in\Gamma(\en TM)$  such that $A^2=\id_{TQ}$, the \mbox{$\pm1$-eigen}subbundles of $A$ are of the same rank and involutive for $[\,\,,\,]_\text{Lie}$, and $g(AX,AY)=-g(X,Y)$. We recall the construction of a Courant algebroid structure on $TQ$ given in  \cite[Prop. 3.13]{SvoASPH}, for $X,Y\in\mathfrak{X}(Q)$ and $X_\pm, Y_\pm$ their projections to $\pm1$-eigenbundles of $A$, by
\begin{align*}
    \la\,\,,\,\ra&\coloneqq g,\! & \rho&\coloneqq \frac{1}{2}(\id_{TQ}+A),\! & [X,Y]&\coloneqq[X_+,Y_+]_\text{Lie}+g^{-1}(L_{X_+}g(Y_-)-\iota_{Y_+}\dif g(X_-)).
\end{align*}
Given a flat connection on $\nabla$ on $M$, the Patterson-Walker metric $g\nb$ together with $A\coloneqq\pr_{\mathcal{H}}-\pr_\mathcal{V}$ is a para-Hermitian structure on $T^*M$. The corresponding Courant algebroid structure on $T(T^*M)$ coincides with $(T(T^*M),g\nb,\rho\nb,[\,\,,\,]\nb)$.

\subsubsection{Lift of an exact curved Courant algebroid}
Consider an exact curved Courant algebroid, which, by Proposition \ref{prop: curved-class}, is isomorphic to $(\mathbb{T}M,\la\,\,,\,\ra_+,\pr_{TM},[\,\,,\,]_{\s{$H$}}^{\s{$T$}})$ for some $T\in\Omega^2(M,TM)$ and $H\in\Omega^3(M)$. Given a connection on $M$, we have, by Proposition \ref{prop: ca-lift-curved-ca}, that its lift is a curved Courant algebroid with curvature
\begin{align*}
 F\nb(X^\text{h}+\alpha^\text{v},Y^\text{h}
+\beta^\text{v})=T(X,Y)^\text{h}-(R\nb(X,Y)^t)^\upsilon.
\end{align*}

\begin{remark}
 If, in particular, $T$ is the torsion of $\nabla$, then $F\nb$ becomes the torsion of the natural lift $\hat{\nabla}$ (the affine connection on $T^*M$) of $\nabla$ given by
\begin{align*}
 \hat{\nabla}_{X^\text{h}}Y^\text{h}&\coloneqq(\nabla_XY)^\text{h}, & \hat{\nabla}_{\alpha^\text{v}}X^\text{h}&\coloneqq0, & \hat{\nabla}_{X^\text{h}}\alpha^\text{v}&\coloneqq(\nabla_X\alpha)^\text{v}, & \hat{\nabla}_{\alpha^\text{v}}\beta^\text{v}&\coloneqq0,
\end{align*}
see \cite[Lem. 5.9.]{SymCartan}.
\end{remark}

 By Corollary \ref{cor: curved-pre-ca} and \eqref{eq: class-pre}, the lift is a pre-Courant algebroid if and only if $F=0$, that is, $T=0$ (that is, the original tuple is pre-Courant) and $\nabla$ is flat. By Theorem \ref{thm: curved-ca}, the lift becomes a Courant algebroid if, in addition, $H$ is closed.

\subsection{Lifts of non-exact tuples}\label{sec: examples} 
We provide additional Courant algebroid lifts.

\subsubsection{Lifts of transitive Courant algebroids}\label{sec: transitive}
 We continue with the case when the original tuple is transitive. Any transitive Courant algebroid is isomorphic to $\mathbb{T}M\oplus \mathcal{G}$ with $\mathcal{G}$ a bundle of quadratic Lie algebras and anchor given by the projection onto $TM$. This isomorphism is called a \textit{dissection} of a transitive Courant algebroid \cite{ChenRCA} and we will regard all transitive Courant algebroids through a dissection. The anchor of a Courant algebroid lift is clearly the projection on the horizontal subbundle and Theorem \ref{thm: curved-ca} yields the following.

\begin{corollary}\label{cor: transitive-ca-lift} The lift of a transitive Courant algebroid (regarded as $\mathbb{T}M\oplus\mathcal{G}$) by a connection $\nabla$ on $E=T^*M\oplus \mathcal{G}$ is Courant if and only if $\nabla$ is flat.
\end{corollary}

 \begin{remark}
 Although a dissection of a transitive Courant algebroid is not unique, any two dissections are isomorphic and it is straightforward to check that so are their lifts. Therefore, we deal with dissected structures without loss of generality.
 \end{remark}

\subsubsection{Lifts of $B_n$-Courant algebroid}\label{sec: Bn} We focus on the special case of a transitive Courant algebroid related to the so-called \textit{$B_n$-generalized geometry} \cite{BarLeib,RubBn}. The \textbf{$B_n$-Courant algebroid} is the tuple consisting of
\begin{itemize}
	\item the bundle $\mathbb{E}=\mathbb{T}M\oplus (M\times\R)$, that is, $E=T^*M\oplus (M\times\R)$,
	\item the pairing $\langle X+\alpha+f,Y+\beta+\cg\rangle \coloneqq \alpha(Y)+\beta(X)+f\cg$,
	\item the projection on the tangent bundle $\rho\coloneqq \mathrm{pr}_{TM}$,	
 	\item the bracket\newline $[X+\alpha+f,Y+\beta+\cg] \coloneqq [X,Y]_\text{Lie}+(L_X\beta-\iota_Y\dif \alpha+f\,\dif \cg) + (X\cg-Yf)$, 
\end{itemize}
 where $X+\alpha+f,Y+\beta+\cg\in\Gamma(\mathbb{E})$. As $\rho(f)=0$ and $f\notin\im\rho^*=T^*M$, this is indeed transitive but not exact.

The lift of the $B_n$-Courant algebroid gives us the \textbf{odd Patterson-Walker metric} $g\nb$ living on the manifold $E$ instead of $T^*M$ and having the signature $(n+1,n)$, where $n\coloneqq \dim M$. In natural local coordinates $(\mathcal{U},\lbrace x^i\rbrace\cup\lbrace p_i\rbrace\cup\lbrace t\rbrace)$, where $\mathcal{U}\coloneqq T^*U\oplus (U\times\R)$, we have that
\begin{align*}
 \rest{g\nb}{\mathcal{U}}=&\,\dif p_i\odot \dif x^i+\dif t\otimes \dif t\\
 &+\big(p_k(\pr^*\Gamma^k_i)+t(\pr^*\Gamma_i)\big)\dif x^i\odot \dif t-\big(p_k(\pr^*\Gamma^k_{ij})-t(\pr^*\Gamma_{ij})\big)\dif x^i\odot\dif x^j\numberthis\label{eq: odd}\\
 &-(p_k(\pr^*\Gamma^k_i)+t(\pr^*\Gamma_i))(p_l(\pr^*\Gamma^l_j)+t(\pr^*\Gamma_j))\dif x^i\odot \dif x^j,
\end{align*}
where $\{ \Gamma^k_{ji}\}\cup \{\Gamma^k_i\rbrace\cup\{\Gamma_{ki}\}\cup\{\Gamma_i\}\subseteq\cCi(U)$ are the coefficients of the vector bundle connection $\nabla$ on $E$ uniquely determined by the relations
\begin{align*}
 \Gamma^k_{ji}\dif x^j&\coloneqq -\nabla_{\partial_{x^i}}\dif x^k, & \Gamma^k_i\cdot 1&\coloneqq \nabla_{\partial_{x^i}}\dif x^k, & \Gamma_{ki}\dif x^k&\coloneqq \nabla_{\partial_{x^i}}1, & \Gamma_i\cdot 1&\coloneqq \nabla_{\partial_{x^i}}1.
\end{align*}

A special case occurs when the connection on $E$ is constructed from a connection $\nabla$ on $M$ and a field of endomorphisms $A\in\Gamma(\en TM)$ by
\begin{equation}\label{eq: Bn-special}
 \nabla^{\s{$A$}}_X(\alpha +f)\coloneqq \nabla_X\alpha+Xf+\alpha(AX).
\end{equation}
In natural coordinates, this corresponds to $\Gamma_{ki}=\Gamma_i=0$ and $\Gamma^k_i=A^k_i$, where $\rest{A}{U}=A^k_i\,\dif x^i\otimes \partial_{x^k}$. The odd Patterson-Walker metric then simplifies as
\begin{equation}
 \begin{split}\label{eq: odd-special}
 \rest{g\nb}{\mathcal{U}}=&\,\dif p_i\odot \dif x^i-p_k(\pr^*\Gamma^k_{ij})\dif x^i\odot\dif x^j\\
 &+ \dif t\otimes \dif t+ p_k(\pr^*A^k_i)\dif x^i\odot \dif t+p_kp_l(\pr^*A^k_iA^l_j)\dif x^i\odot\dif x^j.
 \end{split}
\end{equation}

\begin{remark}
 Coordinate expressions \eqref{eq: odd} and \eqref{eq: odd-special} for the odd Patterson-Walker metric shall be compared with that for the usual Patterson-Walker metric:
 \begin{equation*}
 \rest{g\nb}{T^*U}=\,\dif p_i\odot \dif x^i-p_k(\pr^*\Gamma^k_{ij})\dif x^i\odot\dif x^j.
 \end{equation*}
\end{remark}

By Corollary \ref{cor: transitive-ca-lift}, the Courant algebroid lift by $\nabla^{\s{$A$}}$ is Courant if and only if
\begin{align*}
0= R_{\scalebox{0.45}{$\nabla^A$}}(X,Y)(\alpha+f)=-R\nb(X,Y)^t\alpha+ \alpha((\dif\nb A)(X,Y)+A(T\nb(X,Y)),
\end{align*}
where $\dif\nb\colon\Omega^r(M,TM)\rightarrow\Omega^{r+1}(M,TM)$ is the exterior covariant derivative and $T\nb\in\Omega^2(M,TM)$ is the torsion corresponding to the connection $\nabla$ on $M$. So, $\nabla^{\s{$A$}}$ is flat if and only if $\nabla$ is flat and $(\dif\nb A)(X,Y)=-A(T\nb(X,Y))$.

Therefore, we can use the Courant algebroid lift to characterize \textit{special complex structures} \cite{AleSCM}, that is, pairs $(J,\nabla)$ consisting of an almost complex structure $J$ on $M$ and a torsion-free connection $\nabla$ on $M$ such that $\nabla$ is flat and $\dif\nb J=0$. Conversely, a special complex structure determines a Courant algebroid. 

\begin{corollary}\label{cor: special-complex}
 Consider a pair $(J,\nabla)$ consisting of an almost complex structure $J$ on $M$ and a torsion-free connection $\nabla$ on $M$. The lift of the $B_n$-Courant algebroid by the connection $\nabla^{\s{$J$}}$ is Courant if and only if $(J,\nabla)$ is special complex.
\end{corollary}

\subsubsection{Lifts of quadratic Lie algebras} 
For a tuple $(\mathbb{E},\la\,\,,\,\ra,\rho,[\,\,,\,])$ over $M=\{ *\}$, we have that $TM$ and $\rho$ are trivial, and $(\mathbb{E},[\,\,,\,])$ becomes, by \ref{ca1} and $\ref{ca3}$, a real finite-dimensional Lie algebra together with a pairing $\la\,\,,\,\ra$ satisfying, by \ref{ca2}, $0=\langle [a,b],c\rangle+\langle b,[a,c]\rangle$, that is, a quadratic Lie algebra. This is indeed a one-to-one correspondence.

\begin{example}
Every semi-simple Lie algebra with the \textit{Cartan-Killing form} is a quadratic Lie algebra. Also, every Lie bialgebra (equivalently, a Manin triple) gives a quadratic Lie algebra via its double \cite{DriQG}.
\end{example}

A quadratic Lie algebra $(\mathfrak{g},\la\,\,,\,\ra,[\,\,,\,]_\mathfrak{g})$ can thus be seen as a Courant algebroid on the bundle $T\{*\}\oplus \mathfrak{g}\cong \mathfrak{g}$, hence we can apply the Courant algebroid lift ($\mathbb{E}=E=\mathfrak{g}$ in this case). There is a unique choice of a vector bundle connection on $\mathfrak{g}$ given by $\nabla u=0$ for all $u\in\mathfrak{g}$, which is, moreover, flat. By Corollary \ref{cor: transitive-ca-lift}, we get the Courant algebroid $(T\mathfrak{g}, g,\rho,[\,\,,\,])$. The horizontal lift is the $0$ map in this case, hence, by \eqref{eq: rho-nb-hor}, we have that $\rho=0$, and $g$ together with $[\,\,,\,]$ are determined pointwise by the structure of the quadratic Lie algebra. In particular, we have that the bracket is $\cCi(\mathfrak{g})$-bilinear and $(T_u\mathfrak{g},\la\,\,,\,\ra_u,[\,\,,\,]_u)$ is a quadratic Lie algebra canonically isomorphic to $\mathfrak{g}$ for all $u\in \mathfrak{g}$. Therefore, $(T\mathfrak{g},g,[\,\,,\,])$ is a \textit{bundle of quadratic Lie algebras} where, moreover, all the fibres are isomorphic. 

In fact, it is not difficult to prove that there is a one-to-one correspondence
\begin{equation*}
\left\{
\text{Courant algebroids with } \rho=0\right\}\cong\left\{\text{bundles of quadratic Lie algebras}\right\}.
\end{equation*}

\subsubsection{Courant algebroid induced by a Poisson structure}\label{sec: Poisson}
Every Poisson structure $\pi\in\mathfrak{X}^2(M)$ endows $T^*M$ with the Lie algebroid structure with anchor $\pi\colon T^*M\rightarrow TM$ and bracket $[\alpha,\beta]_\pi\coloneqq L_{\pi(\alpha)}\beta-L_{\pi(\beta)}\alpha-\dif\pi(\alpha,\beta)$, see e.g. \cite{CraLPG}. Choosing the trivial Lie algebroid structure on $TM$ (with zero anchor and zero bracket) makes $(TM, T^*M)$ a Lie bialgebroid, hence a Poisson structure $\pi$ equips $\mathbb{T}M$ with a Courant algebroid structure with pairing $\la\,\,,\,\ra_+$ and anchor $\pi\circ\pr_{T^*M}$. It is regular if and only if the underlying Poisson structure is regular, and it is transitive (and hence exact) if and only if $\pi$ is invertible, that is, $\pi^{-1}$ is a symplectic structure. 

\begin{remark}
 Alternatively, the bracket $[\,\,,\,]$ defined by the Lie bialgebroid construction from $\pi\in\mathfrak{X}^2(M)$ is a derived bracket in the sense of \cite{KosDB}. Using the Poisson differential $\dif _\pi\in\en(\mathfrak{X}^\bullet(M))$ and identifying $X+\alpha\in\Gamma(\mathbb{T}M)$ with $X\wedge +\iota_\alpha\in\en(\mathfrak{X}^\bullet(M))$, we have
\begin{equation*}
 [X+\alpha,Y+\beta]= [[X+\alpha,\dif _\pi]_\mathrm{g},Y+\beta]_\mathrm{g}=L^\pi_\alpha Y-\iota_\beta\dif_\pi X+[\alpha,\beta]_\pi,
 \end{equation*}
 where $L^\pi$ is the Lie derivative associated with the Lie algebroid structure.
\end{remark}

\begin{example}
 A degenerate left-invariant Poisson structure on a Lie group leads to a regular non-transitive Courant algebroid, whereas a linear Poisson structure (a real finite-dimensional Lie algebra) gives a non-regular Courant algebroid.
\end{example}

Given an affine connection $\nabla$ on $M$, by Proposition \ref{prop: ca-lift-curved-ca}, the lift of the Courant algebroid induced by a Poisson structure $\pi$ is the curved Courant algebroid on $T(T^*M)$ with the Patterson-Walker metric, whose curvature is
\begin{equation}\label{eq: Poisson-curvature}
 F\nb(X^\text{h}+\alpha^\text{v}, Y^\text{h}+\beta^\text{v})=-(R\nb(\pi(\alpha),\pi(\beta))^t)^\upsilon.
\end{equation}
By Theorem \ref{thm: curved-ca}, it becomes a Courant algebroid if and only if $R\nb$ vanishes on $\im\pi$, that is, $\nabla$ is flat when restricted to Hamiltonian vector fields.

For instance, given a linear Poisson structure, the lift by the natural \textit{Euclidean connection}, which is flat, is always Courant. On the other hand, we have that even lifts by non-flat connections can be Courant algebroids.

\begin{example}\label{ex: Poisson-Heisenberg}
 Consider the Poisson structure $\pi$ on $\R^3$ given by the \textit{Heisenberg Lie algebra} $\mathfrak{h}_3$, that is, $\pi= z\,\partial_x\wedge \partial_y$. By \eqref{eq: Poisson-curvature}, the lift of the corresponding Courant algebroid by a connection $\nabla$ on $\R^3$ is a Courant algebroid if and only if
 \begin{equation*}
 z^2\,R\nb(\partial_x,\partial_y)=0,
 \end{equation*}
 that is, by continuity, $R\nb(\partial_x, \partial_y)=0$. A concrete non-flat example of such connection is given by declaring $\partial_x$ and $\partial_y$ to be parallel and $\nabla \partial_z\coloneqq f\, \dif z\otimes\partial_z$ for some $f\in\cCi(\R^3)$ non-constant. Indeed, $ R\nb=\dif f\wedge\dif z\otimes\dif z\otimes \partial_z$.
\end{example}

\subsection{Relation to Courant algebroid actions}
We finish by drawing a relation between Courant algebroid lifts and \textit{Courant algebroid actions} \cite[Def. 2.11]{LiBlCAPG}.

\begin{definition}\label{def: ca-action}
 Let $(\mathbb{E},\la\,\,,\,\ra,\rho,[\,\,,\,])$ be a Courant algebroid over $M$. A \textbf{Courant algebroid action on a manifold} $M'$ is a pair $(\Phi,\varrho)$ consisting of $\Phi\colon M'\rightarrow M$ and an $\R$-linear map $\varrho\colon\Gamma(\mathbb{E})\rightarrow\mathfrak{X}(M')$ satisfying
 \begin{align*}
 \Phi_{*q}\varrho(a)&=\rho(a)_{\Phi(q)}, & \varrho(fa)&=(\Phi^*f)\varrho(a), & [\varrho(a),\varrho(b)]_\text{Lie}&=\varrho([a,b]).
 \end{align*}
 The \textbf{stabilizer} of the action at a point $q\in M'$ is the kernel of the map
 \begin{align*}
 \varrho_q\colon \mathbb{E}_{\Phi(q)}&\rightarrow T_qM'\\
 \text{a}&\mapsto\varrho(a)_q,
 \end{align*}
 where $a$ is an arbitrary section of $\mathbb{E}$ such that $a_{\Phi(q)}=\text{a}$.
\end{definition}

The following theorem shows that Courant algebroid lifts offer a large class of examples of Courant algebroid actions.

\begin{theorem}\label{thm: nabla-ca-action}
Let $(\mathbb{E},\la\,\,,\,\ra,\rho,[\,\,,\,])$ such that $\mathbb{E}=TM\oplus E$ for some vector bundle $\pr\colon E\rightarrow M$ be a Courant algebroid, and let $\nabla$ be a vector bundle connection on $E$. The pair $(\pr,\rho\nb\circ\phi )$ is a Courant algebroid action of $(\mathbb{E},\la\,\,,\,\ra,\rho,[\,\,,\,])$ on $E$ if and only if the corresponding Courant algebroid lift is a Courant algebroid. If this is the case, the stabilizer of the action at an arbitrary point of $E$ is coisotropic. 
\end{theorem}

\begin{proof}
By \eqref{eq: rho-nb-hor} and Lemma \ref{lem: hor-ver}, we obtain that
\begin{align*}
\pr_{*q}\rho\nb(\phi a)&=\pr_{*q}\rho(a)^\text{h}=\rho(a)_{\pr(q)},\\
 \rho\nb(\phi (fa))&=\rho(fa)^\text{h}=(\pr^*f)\rho(a)^\text{h}=(\pr^*f)\rho\nb(\phi a),
\end{align*}
that is, the pair $(\pr,\rho\nb\circ\phi )$ satisfies the first two axioms for a Courant algebroid action on $E$. By \eqref{eq: rho-nb-hor}, the last axiom takes the form:
\begin{equation*}
 0=\rho\nb(\phi [a,b])-[\rho\nb(\phi a),\rho\nb(\phi b)]_\text{Lie}=\rho([a,b])^\text{h}-[\rho(a)^\text{h},\rho(b)^\text{h}]_\text{Lie},
\end{equation*}
that is, since $(\mathbb{E},\la\,\,,\,\ra,\rho,[\,\,,\,])$ satisfies \ref{ca6} and Lemma \ref{lem: hor-ver-com}, equivalent to
\begin{equation*}
 0=[\rho(a),\rho(b)]_\text{Lie}^\text{h}-[\rho(a)^\text{h},\rho(b)^\text{h}]_\text{Lie}=R\nb(\rho(a),\rho(b))^\upsilon
\end{equation*}
The result follows directly from the injectivity of $(\;)^\upsilon$ and Theorem \ref{thm: curved-ca}.

As $\varrho=\rho\nb\circ\phi$, we have that $\varrho_q=\rho\nb\circ\phi ^q$ for every $q\in E$ and, by \eqref{eq: kernels}, the stabilizer of the action at $q$ is thus $\ker\rho_{\pr(q)}$. Proposition \ref{prop: coisotropic} and the fact that $(\mathbb{E},\la\,\,,\,\ra,\rho,[\,\,,\,])$ is a Courant algebroid give that all stabilizers are coisotropic.
\end{proof}

The ‘only if’ part of the theorem matches \cite[Thm. 2.12]{LiBlCAPG} dealing with the question of when a Courant algebroid action $(\Phi,\varrho)$ of $(\mathbb{E},\la\,\,,\,\ra,\rho,[\,\,,\,])$ on a manifold $M'$ induces a Courant algebroid structure on the pull-back bundle $\Phi^*\mathbb{E}\rightarrow M'$. This Courant algebroid  structure is canonically isomorphic to the one given by the Courant algebroid lift construction on $TE$. 
The most interesting part of Theorem \ref{thm: nabla-ca-action} is the ‘if’ part, which provides a large class of examples of Courant algebroid actions. The following corollaries are direct consequences of Theorem \ref{thm: nabla-ca-action} applied to examples from Sections \ref{sec: PW-Ca} and \ref{sec: examples}.

\begin{corollary}
 A transitive Courant algebroid (regarded as $\mathbb{T}M\oplus\mathcal{G}$) acts on $E\coloneqq T^*M\oplus\mathcal{G}$ by flat connections on $E$.
\end{corollary}

In particular, we have that:

\begin{corollary}\label{cor: exact-action}
 An exact Courant algebroid (regarded as $\mathbb{T}M$) acts on $T^*M$ by flat affine connections on $M$.
\end{corollary}

\begin{corollary}\label{cor: special-complex-action}
 The $B_n$-Courant algebroid (Section \ref{sec: Bn}) acts on \mbox{$T^*M\oplus (M\times\R)$} by special complex structures on $M$.
\end{corollary}

On the other hand, we have also that:

\begin{corollary}\label{cor: Poisson-action}
 The Courant algebroid induced by a Poisson structure on $M$ (Section \ref{sec: Poisson}) acts on $T^*M$ by affine connections on $M$ that are flat when restricted to Hamiltonian vector fields.
\end{corollary}

\appendix

\section{Some proofs of Sections \ref{sec: hierarchy} and \ref{sec: exact-curved}}\label{app:proofs}

\begin{proof}[Proof of Lemma \ref{lem: ca-4-5}]
 By \ref{ca2}, $ \rho(a)\la b,fc\ra=\la[a,b],fc\ra+\la b,[a,fc]\ra$, hence
 \begin{equation*}
 (\rho(a)f)\la b,c\ra+f(\rho(a)\la b,c\ra)=f\la[a,b],c\ra+\la b,[a,fc]\ra.
 \end{equation*}
 Using \ref{ca2} once more yields $ \la b,(\rho(a)f)c+f[a,c]\ra=\la b,[a,fc]\ra$ and the \mbox{non-deg}eneracy of the pairing proves \ref{ca4}.
 
 By \ref{ca1}, we have that
 \begin{equation*}
 [a,fb]+[fb,a]=\rho^*\dif\la a,fb\ra=f\rho^*\dif\la a,b\ra+\la a,b\ra\rho^*\dif f.
 \end{equation*}
 By \ref{ca4},
 \begin{equation*}
 [fb,a]=-(\rho(a)f)b-f[a,b]+f\rho^*\dif\la a,b\ra+\la a,b\ra\rho^*\dif f.
 \end{equation*}
 Using \ref{ca1} once more, we get that $[fb,a]=-(\rho(a)f)b+f[b,a]+\la a,b\ra\rho^*\dif f$.
\end{proof}

We prove here a helpful technical lemma.

\begin{lemma}\label{lem: technical}
 For a metric algebroid $(E,\la\,\,,\,\ra,\rho,[\,\,,\,])$ and $F$ as in \eqref{eq: ante-curvature}, we have
 \begin{align*}
 \la b,[a,\rho^*\alpha]\ra
 &=\la b,\rho^*L_{\rho(a)}\alpha\ra-\alpha(F(a,b)), & [\rho^*\alpha,a]&=-[a,\rho^*\alpha]+\rho^*\dif(\alpha(\rho(a))).
 \end{align*}
\end{lemma}

\begin{proof}
 By \ref{ca2}, we get
 \begin{align*}
 \la b,[a,\rho^*\alpha]\ra
 &=-\la[a,b],\rho^*\alpha\ra+\rho(a)\la b,\rho^*\alpha\ra=-\alpha(\rho([a,b]))+\rho(a)\alpha(\rho(b))\\
 &=-\alpha(F(a,b)+[\rho(a),\rho(b)]_\text{Lie})+\rho(a)\alpha(\rho(b)) \\
 &=(L_{\rho(a)}\alpha)(\rho(b))-\alpha(F(a,b))=\la b,\rho^*L_{\rho(a)}\alpha\ra-\alpha(F(a,b)).
 \end{align*}
 Whereas \ref{ca1} gives that
 \begin{equation*}
 [\rho^*\alpha,a]=-[a,\rho^*\alpha]+\rho^*\dif\la a,\rho^*\alpha\ra=-[a,\rho^*\alpha]+\rho^*\dif(\alpha(\rho(a))).
 \end{equation*}
\end{proof}

\begin{proof}[Proof of Lemma \ref{lem: ca-7-8}]
It follows directly from Lemma \ref{lem: technical} and that a metric algebroid is pre-Courant if and only if $F=0$.
\end{proof}

\begin{proof}[Proof of Lemma \ref{lem: ca-13}]
Let $(E,\la\,\,,\,\ra,\rho,[\,\,,\,])$ be a metric algebroid satisfying \ref{ca13}. By \ref{ca1} and \ref{ca13}, we get that $F$ is skew-symmetric:
\begin{equation*}
 F(a,a)=\rho([a,a])=\frac{1}{2}\rho(\rho^*\dif\la a,a\ra)=0.
\end{equation*}
The $\cCi(M)$-bilinearity of $F$ then follows straightforwardly from \ref{ca4}.

On the other hand, let $(E,\la\,\,,\,\ra,\rho,[\,\,,\,])$ be an ante-Courant algebroid. By \ref{ca1},
 \begin{equation*}
 \rho(\rho^*\dif\la a,b\ra)=\rho([a,b]+[b,a])=[\rho(a),\rho(b)]_\text{Lie}+F(a,b)+[\rho(b),\rho(a)]_\text{Lie}+F(b,a)=0.
 \end{equation*}
Therefore, we have the following:
 \begin{equation*}
 0=\rho(\rho^*\dif\la a,fb\ra)=\la a,b\ra\rho(\rho^*\dif f)+f\rho(\rho^*\dif\la a,b\ra)=\la a,b\ra\rho(\rho^*\dif f).
 \end{equation*}
By the fact that exact $1$-forms locally generate $\Omega^1(M)$, we arrive at $\rho\circ\rho^*=0$.
\end{proof}

\begin{proof}[Proof of Proposition \ref{prop: coisotropic}]
 The definition of $\rho^*$ yields $\la \rho^*\alpha,\rho^*\beta\ra=\alpha(\rho(\rho^*\beta))$, which shows that \ref{ca13} is clearly equivalent to $\im\rho^*$ being pointwise isotropic. On the other hand, for all $m\in M$ and $\text{a}\in\ker\rho_m$ we have that $\la \text{a},\rho^*\zeta\ra=\zeta(\rho(\text{a}))=0$ for all $\zeta\in T^*_mM$, that is, $\im\rho^*_m\leq (\ker\rho_m)^\perp$ and by the dimensional reasons we thus have $\im\rho^*_m=(\ker\rho_m)^\perp$. As $\im\rho^*_m$ being isotropic means that $\im\rho^*_m\leq (\im\rho^*)^\perp$ and $\ker\rho_m$ being coisotropic means that $(\ker\rho_m)^\perp\leq \ker\rho_m$, the result follows.
\end{proof}

\begin{proof}[Proof of Lemma \ref{lem: ca-10-11}]
 It follows directly from Lemma \ref{lem: technical}.
\end{proof}

\begin{proof}[Proof of Lemma \ref{lem: reg-ca-10-11}]
 By Proposition \ref{prop: reg-trans-curved-ca} and Lemma \ref{lem: technical}, we get that $(\mathbb{E},\la\,\,,\,\ra,\rho,[\,\,,\,])$ is curved Courant if and only if there is $T\in\Omega^2(M,TM)$ such~that
 \begin{equation*}
 \la b,[a,\rho^*\alpha]\ra
     =\la b,\rho^*L_{\rho(a)}\alpha\ra-\alpha(T(\rho(a),\rho(b)))=\la b,\rho^*(L_{\rho(a)}-(\iota_{\rho(a)}T)^t)\alpha\ra,
 \end{equation*}
 hence the equivalence to the first statement follows. The equivalence to the second statement follows immediately from Lemma \ref{lem: technical}.
\end{proof}

\begin{proof}[Proof of Proposition \ref{prop: curved-class}] 
 Let $(\mathbb{E},\la\,\,,\,\ra,\rho,[\,\,,\,])$ be an exact curved Courant algebroid over $M$ with the curvature $T\in\Omega^2(M,TM)$ and choose an isotropic splitting $s\colon TM\rightarrow \mathbb{E}$ of the short exact sequence \eqref{eq: ex-seq}. We define the vector bundle morphism $\Phi\colon \mathbb{T}M\rightarrow \mathbb{E}$ over the identity by
 \begin{equation*}
 \Phi(u+\zeta)\coloneqq s(u)+\rho^*\zeta,
 \end{equation*}
 for $u+\zeta\in\mathbb{T}M$. As $\mathbb{E}=\im s\oplus \im\rho^*$, and, moreover, both $s$ and $\rho^*$ are injective, we have that $\Phi$ is a vector bundle isomorphism. It follows by the isotropy of the splitting, \ref{ca1} and \ref{ca2} that $H\in\Omega^3(M)$ is defined by the formula
\begin{equation*}
 H(X,Y,Z)\coloneqq\la [s(X),s(Y)],s(Z)\ra.
\end{equation*}

Let us now show that $\Phi$ is indeed an isomorphism between $(\mathbb{T}M,\la\,\,,\,\ra_+,\pr_{TM},[\,\,,\,]_{\s{$H$}}^{\s{$T$}})$ and the curved Courant algebroid we started with. By the isotropy of $s$ and $\im\rho^*$ (the latter follows from Proposition \ref{prop: coisotropic}),
\begin{equation*}
 \la \Phi(X+\alpha),\Phi(X+\alpha)\ra=\la s(X)+\rho^*\alpha,s(X)+\rho^*\alpha\ra=2\la s(X),\rho^*\alpha\ra.
\end{equation*}
By the definition of $\rho^*$ and that $s$ is a splitting, we get 
\begin{equation*}
 \la \Phi(X+\alpha),\Phi(X+\alpha)\ra=2\alpha(\rho(s(X)))=2\alpha(X)=\la X+\alpha,X+\alpha\ra_+,
\end{equation*}
hence, by polarization, $\Phi$ intertwines the pairings. Analogously, we have
\begin{equation*}
 \rho(\Phi(X+\alpha))=\rho(s(X)+\rho^*(\alpha))=X=\pr_{TM}(X+\alpha)
\end{equation*}
for the anchors. Finally, the linearity of the bracket gives that
\begin{align*}
 [\Phi(X+\alpha),\Phi(Y+\beta)]=[s(X),s(Y)]+[s(X),\rho^*\beta]+[\rho^*\alpha,s(Y)]+[\rho^*\alpha,\rho^*\beta].
\end{align*}
By Lemma \ref{lem: reg-ca-10-11} and the fact that $s$ is a splitting, we get
\begin{align*}
 [s(X),\rho^*\beta]&=\rho^*(L_{\rho(s(X))}-(\iota_{\rho(s(X))}T)^t)\beta=\Phi(L_X-(\iota_XT)^t)\beta=\Phi(L^{\s{$T$}}_X\beta),\\
 [\rho^*\alpha,s(Y)]&=\rho^*(-\iota_{\rho(s(Y))}\circ\dif+(\iota_{\rho(s(Y))}T)^t)\alpha=\Phi(-\iota_Y\circ\dif+(\iota_YT)^t)\alpha\\
 &=\Phi(-\iota_Y\dif^{\s{$T$}}\alpha).
\end{align*}
Lemma \ref{lem: reg-ca-10-11} together with \ref{ca13} yield $[\rho^*\alpha,\rho^*\beta]=0$. It remains to deal with the term $[s(X),s(Y)]$. As $\Phi$ is a vector bundle isomorphism, there is a unique $Z+\eta\in\Gamma(\mathbb{T}M)$ such that $[s(X),s(Y)]=\Phi(Z+\eta)=s(Z)+\rho^*\eta$. By \ref{ca9},
\begin{align*}
 \rho([s(X),s(Y)])&=[\rho(s(X)),\rho(s(Y))]_\text{Lie}+T(\rho(s(X)),\rho(s(Y)))\\
 &=[X,Y]_\text{Lie}+T(X,Y)=[X,Y]^{\s{$T$}},
\end{align*}
hence $Z=[X,Y]^{\s{$T$}}$. On the other hand, the isotropy of $s$ implies that
\begin{equation*}
 H(X,Y,Z)=\la [s(X),s(Y)],s(Z)\ra=\la\rho^*\eta,s(Z)\ra=\eta(\rho(s(Z)))=\eta(Z),
\end{equation*}
that is, $\eta=\iota_Y\iota_XH$. Altogether, we have proved that
\begin{equation*}
 [\Phi(X+\alpha),\Phi(Y+\beta)]=\Phi([X+\alpha,Y+\beta]_{\s{$H$}}^{\s{$T$}})
\end{equation*}
and the result follows.
\end{proof}

\begin{proof}[Proof of Lemma \ref{lem: curved-class}]
 Every isomorphism between two metric algebroids $(\mathbb{T}M,\la\,\,,\,\ra_+,\pr_{TM},[\,\,,\,])$ and $(\mathbb{T}M,\la\,\,,\,\ra_+,\pr_{TM},[\,\,,\,]')$ is necessarily of the form of a $B$-transform, that is, $\e^B\colon (X+\alpha)\mapsto X+\alpha+\iota_XB$ for some $B\in\Omega^2(M)$, see e.g. \cite{GuaGCG}. If, in addition, we have that $[\,\,,\,]=[\,\,,\,]_{\s{$H$}}^{\s{$T$}}$ and $[\,\,,\,]'=[\,\,,\,]_{\s{$H'$}}^{\s{$T'$}}$, the $B$-transform $\e^B$ is an isomorphism if and only if
 \begin{equation*}
 [\e^B(X+\alpha),\e^B(Y+\beta)]_{\s{$H'$}}^{\s{$T'$}}=\e^B([X+\alpha,Y+\beta]^{\s{$T$}}_{\s{$H$}}).
 \end{equation*}
 Explicitly, we have that
 \begin{equation}\label{eq: curved-ca-class}
 \begin{split}
 [X,Y]^{\s{$T'$}}+L^{\s{$T'$}}_X\beta-\iota_Y&\dif^{\s{$T'$}}\alpha+L_X^{\s{$T'$}}\iota_YB-\iota_Y\dif^{\s{$T'$}}\iota_XB+\iota_Y\iota_XH'\\
 &=[X,Y]^{\s{$T$}}+L^{\s{$T$}}_X\beta-\iota_Y\dif^{\s{$T$}}\alpha+\iota_Y\iota_XH+\iota_{[X,Y]^{\s{$T$}}}B.
 \end{split}
 \end{equation}
 The vector field part of the equation reads as $[X,Y]^{\s{$T'$}}=[X,Y]^{\s{$T$}}$, which is clearly equivalent to $T=T'$. Using this on the $1$-form part of \eqref{eq: curved-ca-class} yields
 \begin{equation*}
 L_X^{\s{$T$}}\iota_YB-\iota_Y\dif^{\s{$T$}}\iota_XB+\iota_Y\iota_XH'=\iota_Y\iota_XH+\iota_{[X,Y]^{\s{$T$}}}B.
 \end{equation*}
 Equivalently,
 \begin{equation*}
 \iota_Y\iota_X(H'-H)=(\iota_{[X,Y]^{\s{$T$}}}-L_X^{\s{$T$}}\iota_Y+\iota_Y\dif^{\s{$T$}}\iota_X)B=\iota_Y(-L^{\s{$T$}}_X+\dif^{\s{$T$}}\iota_X)B=-\iota_Y\iota_X\dif^{\s{$T$}} B
 \end{equation*}
 and the result follows.
\end{proof}

\section{Various lifts to the total space of a vector bundle}\label{app: lifts}
A vector bundle connection $\nabla$ on a vector bundle $\pr\colon E\rightarrow M$ yields a decomposition into the horizontal subbundle $\mathcal{H}\nb$ and the vertical subbundle $\mathcal{V}$:
\begin{equation*}
 TE= \mathcal{H}\nb\oplus\mathcal{V}.
\end{equation*}
The canonical isomorphisms $\mathcal{H}\nb\cong \pr^*TM$ and $\mathcal{V}\cong \pr^*E$ determine, via pull-back, the \textbf{horizontal and vertical lifts}:
\begin{align*}
 (\;)^\text{h}\colon \mathfrak{X}(M)&\rightarrow\Gamma(\mathcal{H}\nb)\subseteq\mathfrak{X}(E) & (\;)^\text{v}\colon \Gamma(E)&\rightarrow\Gamma(\mathcal{V})\subseteq\mathfrak{X}(E)\\
 X&\mapsto \pr^*X, & \varphi&\mapsto \pr^*\varphi,
\end{align*}
where $(\;)^\text{h}$ depends on the choice of $\nabla$, while $(\;)^\text{v}$ is canonical.

Consider a local chart $(U,\lbrace x^i\rbrace)$ on the base manifold $M$ and a local frame $(U,\lbrace e_\mu\rbrace)$ for $E$. In the corresponding local chart $(\mathcal{U}\coloneqq\rest{E}{U},\lbrace x^i\rbrace\cup\lbrace v^\mu\rbrace)$ on $E$, the horizontal lift of $X\in\mathfrak{X}(M)$ and the vertical lift of $\varphi\in\Gamma(E)$ are given by:
\begin{align*}
 &\rest{X}{U}=X^i\partial_{x^i}, & &\rest{X^\text{h}}{\mathcal{U}}=(\pr^*X^i)(\partial_{x^i}-v^\nu(\pr^*\Gamma^\mu_{\nu i})\partial_{v^\mu}),\\
 &\rest{\varphi}{U}=\varphi^\mu e_\mu, & &\rest{\varphi^\text{v}}{\mathcal{U}}=(\pr^*\varphi^\mu)\partial_{v^\mu},
\end{align*}
where $\lbrace\Gamma^\mu_{\nu j}\rbrace\subseteq\cCi(U)$ are uniquely determined by the relations $\Gamma^\mu_{\nu i}e_\mu=\nabla_{\partial_{x^i}} e_\nu$. 

We continue by introducing the \textbf{vertical lift of a dual section}
\begin{equation*}
 (\;)^v\colon \Gamma(E^*)\rightarrow\cCi(E)
\end{equation*}
by the formula $\xi^v(q)\coloneq \xi_{\pr(q)}(q)$ for all $q\in E$. Locally, we get
\begin{align*}
 \rest{\xi}{U}&=\xi_\mu e^\mu, & \rest{\xi^v}{\mathcal{U}}&=(\pr^*\xi_\mu) v^\mu,
\end{align*}
where $(U,\{e^\mu\})$ is the dual frame to $(U,\{e_\mu\})$. It is easy to check that the direct analogue of \cite[Prop. 1]{Patlift} is true.

\begin{proposition}\label{prop: char-vf}
A vector field on $E$ is fully determined by its action
on vertical lifts of all dual sections of $E$.
\end{proposition}

In particular, by using local coordinates, we immediately see that the horizontal and vertical lifts are fully determined by the following relations:
\begin{align}\label{eq: hor-ver}
X^\text{h}\xi^v&=(\nabla_X\xi)^v, & \varphi^\text{v}\xi^v&=\pr^*\xi(\varphi),
\end{align}
where we use the same symbol $\nabla$ for the dual connection on $E^*$.

Another useful lift is the \textbf{vertical lift of a field of endomorphisms}
\begin{equation*}
 (\;)^\upsilon\colon\Gamma(\en E)\rightarrow \mathfrak{X}(E)
\end{equation*}
defined by prescribing its action on the vertical lift of a dual section:
\begin{equation*}
 A^\upsilon\xi^v\coloneqq (A^t\xi)^v.
\end{equation*}
Locally, we have that
\begin{align*}
 \rest{A}{U}&=A^\mu_\nu\,e^\nu\otimes e_\mu, & &\rest{A^\upsilon}{\mathcal{U}}=v^\nu(\pr^*A^\mu_\nu)\partial_{v^\mu}.
\end{align*}

We finish by deriving several useful properties of these lifts.

\begin{lemma}\label{lem: hor-ver}
For $f\in\cCi(M)$, $X\in\mathfrak{X}(M)$ and $\varphi\in\Gamma(E)$,
 \begin{align*}
 (fX)^\emph{h}&=(\pr^*f) X^\emph{h},& (f\varphi)^\emph{v}&=(\pr^*f) \varphi^\emph{v}, & X^\emph{h}\pr^*f&=\pr^*(Xf), & \varphi^\emph{v}
\pr^*f&=0. 
\end{align*}
\end{lemma}

\begin{proof}
 It follows by a straightforward calculation in a local chart.
\end{proof}

We recall that every vector bundle connection has the associated curvature \mbox{$2$-form} $R\nb\in\Omega^2(M, \en E)$ given by $R\nb(X,Y)\varphi\coloneqq \nabla_X\nabla_Y\varphi-\nabla_Y\nabla_X\varphi-\nabla_{[X,Y]_\text{Lie}}\varphi$.

\begin{lemma}\label{lem: hor-ver-com}
 For $X,Y\in\mathfrak{X}(M)$ and $\varphi,\psi\in\Gamma(E)$,
 \begin{align*}
 [X^\emph{h}, Y^\emph{h}]_\emph{Lie}&=[X,Y]_\emph{Lie}^\emph{h}-R\nb(X,Y)^\upsilon, & [X^\emph{h}, \varphi^\emph{v}]_\emph{Lie}&=(\nabla_X\varphi)^\emph{v}, & [\varphi^\emph{v}, \psi^\emph{v}]_\emph{Lie}&=0.
\end{align*}
\end{lemma}

\begin{proof}
 By \eqref{eq: hor-ver} and Lemma \ref{lem: hor-ver}, we have, for $\xi\in\Gamma(E^*)$, that
\begin{align*}
 [X^\text{h}, Y^\text{h}]_\text{Lie}\,\xi^v&=(\nabla_X\nabla_Y\xi-\nabla_Y\nabla_X\xi)^v=(\nabla_{[X,Y]_\text{Lie}}\xi-R_\nabla(X,Y)^t\xi)^\upsilon,\\
 &=([X,Y]_\text{Lie}^\text{h}-R_\nabla(X,Y)^\upsilon)\xi^v\\
[X^\text{h},\varphi^\text{v}]_\text{Lie}\,\xi^v&=X^\text{h}\pr^*\xi(\varphi)-\pr^*(\nabla_X\xi)(\varphi)=\pr^*\xi(\nabla_X\varphi)=(\nabla_X\varphi)^\text{v}\xi^v,\\
[\varphi^\text{v}, \psi^\text{v}]_\text{Lie}\,\xi^v&=\varphi^\text{v}\pr^*\xi(\psi)-\psi^\text{v}\pr^*\xi(\varphi)=0.
\end{align*}
The result follows from Proposition \ref{prop: char-vf}.
\end{proof}

\bibliographystyle{alpha}\bibliography{refs}

\end{document}